\documentclass[30pt]{article}
\usepackage{amsfonts}
\usepackage{mathrsfs}
\usepackage{amsmath}
\usepackage{amssymb}
\usepackage{graphicx}
\usepackage{geometry}
\usepackage{epic}
\usepackage{epstopdf}
\usepackage{color}
\usepackage{cite}
\usepackage{mathtools}
\usepackage{amsthm}
\renewcommand{\paragraph}{\roman{paragraph}}

\usepackage{bm}
\usepackage{hyperref}
\usepackage{tikz}
\usepackage{enumerate}
\usetikzlibrary{arrows,shapes,positioning}
\usetikzlibrary{decorations.markings}
\tikzstyle arrowstyle=[scale=1]
\tikzstyle directed=[postaction={decorate,decoration={markings, mark=at position .65 with {\arrow[arrowstyle]{stealth}}}}]
\tikzstyle reverse directed=[postaction={decorate,decoration={markings, mark=at position .65 with {\arrowreversed[arrowstyle]{stealth};}}}]

\newtheorem{theorem}{Theorem}[section]
\newtheorem{corollary}[theorem]{Corollary}

\newtheorem{question}[theorem]{Question}

\newtheorem{lemma}[theorem]{Lemma}

\newtheorem{Remark}[theorem]{Remark}

\usepackage{enumerate}
\usepackage{amssymb}
\usepackage{graphics}
\usepackage{graphicx}
\usepackage{subfigure}
\usepackage{epsfig}
\usepackage{amsmath}
\usepackage{bm}
\usepackage{lineno}
\usepackage{url}
\usepackage{float}
\usepackage{color}

\begin{document}

\title{On an inverse problem of the Erd\H{o}s-Ko-Rado type theorems}

\author{Xiangliang Kong and Gennian Ge
\thanks{ X. Kong ({\tt 2160501011@cnu.edu.cn}) and G. Ge ({\tt gnge@zju.edu.cn}) are with the School of Mathematical Sciences, Capital Normal University, Beijing 100048, China. The research of G. Ge was supported by the National Natural Science Foundation of China under Grant No. 11971325 and Beijing Scholars Program.}
}

\date{}
\maketitle

\begin{abstract}
  A family of subsets $\mathcal{F}\subseteq {[n]\choose k}$ is called intersecting if any two of its members share a common element. Consider an intersecting family, a direct problem is to determine its maximal size and the inverse problem is to characterize its extremal structure and its corresponding stability. The famous Erd\H{o}s-Ko-Rado theorem answered both direct and inverse problems and led the era of studying intersection problems for finite sets.

  In this paper, we consider the following quantitative intersection problem which can be viewed an inverse problem for Erd\H{o}s-Ko-Rado type theorems: For $\mathcal{F}\subseteq {[n]\choose k}$, define its \emph{total intersection} as $\mathcal{I}(\mathcal{F})=\sum_{F_1,F_2\in \mathcal{F}}|F_1\cap F_2|$. Then, what is the structure of $\mathcal{F}$ when it has the maximal total intersection among all families in ${[n]\choose k}$ with the same family size?

  Using a pure combinatorial approach, we provide two structural characterizations of the optimal family of given size that maximizes the total intersection. As a consequence, for $n$ large enough and $\mathcal{F}$ of proper size, these characterizations show that the optimal family $\mathcal{F}$ is indeed $t$-intersecting ($t\geq 1$). To a certain extent, this reveals the relationship between properties of being intersecting and maximizing the total intersection. Also, we provide an upper bound on $\mathcal{I}(\mathcal{F})$ for several ranges of $|\mathcal{F}|$ and determine the unique optimal structure for families with sizes of certain values.

  \smallskip


\end{abstract}

\section{Introduction}

For a positive integer $n$, let $[n]$ denote the set of the first $n$ positive integers, $[n]=\{1,2,\ldots,n\}$. Let $2^{[n]}$ and ${[n]\choose k}$ denote the power set and the collection of all $k$-element subsets of $[n]$, respectively. $\mathcal{F}\subseteq 2^{[n]}$ is called a family of subsets, and moreover $k$-uniform, if $\mathcal{F}\subseteq {[n]\choose k}$. A family is called \emph{intersecting} if any two of its members share at least one common element. In 1961, Erd\H{o}s, Ko and Rado published the following classic result.
\begin{theorem}(Erd\H{o}s-Ko-Rado \cite{EKR})\label{EKR}
Let $n>k>t>0$ be integers and let $\mathcal{F}\subseteq {[n]\choose k}$ satisfy $|F\cap F'|\geq t$ for all $F,F'\in \mathcal{F}$. Then the following holds:
\begin{itemize}
  \item (i) If $t=1$ and $n\geq 2k$, then
        \begin{equation}\label{EKReq1}
        |\mathcal{F}|\leq {n-1\choose k-1}.
        \end{equation}
  \item (ii) If $t\geq 2$ and $n> n_0(k,t)$, then
        \begin{equation}\label{EKReq2}
        |\mathcal{F}|\leq {n-t\choose k-t}.
        \end{equation}
\end{itemize}
\end{theorem}
As one of the most fundamental results in extremal set theory, this theorem has inspired a great number of extensions and variations. Such as studies of cross-intersecting families (for examples, see \cite{FT92,FK17,WZ11,WZ13}), $L$-intersecting families (for examples, see \cite{FF84,FOT96,Snevily03,MR2014}), intersection problems on families of subspaces and permutations (for examples, see \cite{FW1986,CP2010,EFP2011,CK03,Ellis12,Ellis14}), etc. For readers interested in other extensions, we recommend Frankl and Tokushige's excellent survey \cite{FT2016} and references therein.

Following the path led by Erd\H{o}s, Ko and Rado, most of these extensions and variations concerned problems of a same type of flavour: Consider a family (or families) of subsets, subspaces, or permutations with a certain kind of intersecting property, how large can this family (or these families) be? Since the intersecting property naturally leads to a clustering structure of the family, therefore, the size of the extremal family can not be very large and these kinds of questions are well asked.

For such problems, once we determine the maximal size of the family with given intersecting property, an immediate inverse problem is to characterize the structure of the extremal family. For a simple example, as is shown in \cite{EKR}, the \emph{full $1$-star}, defined as the family consisting of all $k$-sets in ${[n]\choose k}$ containing a fixed element, is proved to be the only structure for intersecting families that achieve the equality in (\ref{EKReq1}) when $n>2k$. Starting from this, the stability and supersaturation for extremal families are then well worth studying. In recent years, there have been a lot of works concerning this kind of inverse problems, for examples, see \cite{BDLST2019,DGS2016,BDDLS2015,KKK2012,Ellis2011,FKR2016,Russell2012,GS2020,DT2016}.

In this paper, with the same spirit, we consider an inverse problem for intersecting families of subsets from another point of view. Instead of being intersecting, we assume that the family possesses a certain property that ``maximizes'' the intersections quantitatively.

To state the problem formally, first, we introduce the notion \emph{total intersection} of a family. Consider a family $\mathcal{F}\subseteq {[n]\choose k}$, for any fixed $A\in{[n]\choose k}$, we denote $\mathcal{I}(A,\mathcal{F})=\sum_{B\in\mathcal{F}}|A\cap B|$ as the \emph{total intersection} between $A$ and $\mathcal{F}$. Then, by a simple double counting argument, we have $\mathcal{I}(A,\mathcal{F})=\sum_{x\in A}|\mathcal{F}(x)|$. Define the \emph{total intersection of $\mathcal{F}$} as $\mathcal{I}(\mathcal{F})=\sum_{A\in\mathcal{F}}\mathcal{I}(A,\mathcal{F})$, therefore, we have the following identity:
\begin{equation}\label{basic_id}
\mathcal{I}(\mathcal{F})=\sum_{A\in\mathcal{F}}\sum_{B\in\mathcal{F}}|A\cap B|=\sum_{x\in[n]}|\mathcal{F}(x)|^2.
\end{equation}
Moreover, for $\mathcal{F}\subseteq {{[n]\choose k}}$, we denote
\begin{equation}\label{basic_id2}
\mathcal{MI}(\mathcal{F})=\max_{\mathcal{G}\subseteq {[n]\choose k}, |\mathcal{G}|=|\mathcal{F}|}\mathcal{I}(\mathcal{G}).
\end{equation}
as the maximal total intersection among all families in ${[n]\choose k}$ with the same size of $\mathcal{F}$.

Certainly, the value of $\mathcal{I}(\mathcal{F})$ reveals the level of intersections among the members of $\mathcal{F}$: the larger $\mathcal{I}(\mathcal{F})$ is, the more intersections there will be in $\mathcal{F}$. Noted that being intersecting also indicates that $\mathcal{F}$ possesses a large amount of intersections, therefore, it is natural to ask the relationship between being intersecting and having large $\mathcal{I}(\mathcal{F})$. Starting from the simplest case, first and foremost, we have the following question:
\begin{question}\label{question0}
Consider a family of subsets $\mathcal{F}$ with size $|\mathcal{F}|={n-1\choose k-1}$, if $\mathcal{I}(\mathcal{F})=\mathcal{MI}(\mathcal{F})$, is $\mathcal{F}$ an intersecting family? Or, if $\mathcal{F}$ is an intersecting family, do we have $\mathcal{I}(\mathcal{F})=\mathcal{MI}(\mathcal{F})$?
\end{question}
Since when $n>2k$, the only structure of the intersecting family with size ${n-1\choose k-1}$ is the full $1$-star, thus, Question \ref{question0} actually asks whether full $1$-star maximizes the total intersection among all $k$-uniform families of size ${n-1\choose k-1}$ and if so, whether it is the only extremal structure. Noticed for $t\geq 1$ and $n$ large enough, \emph{full $t$-star}, defined as the family consisting of all $k$-sets in ${[n]\choose k}$ containing $t$ fixed elements is also the unique extremal structure for $t$-intersecting families. Therefore, Question \ref{question0} can be extended as:
\begin{question}\label{question1}
For $t\geq 1$ and $n$ large enough, let $\mathcal{F}_0\subseteq {[n]\choose k}$ be a full $t$-star, do we have $\mathcal{I}(\mathcal{F}_0)=\mathcal{MI}(\mathcal{F}_0)$? If so, is full $t$-star the only structure of families of size ${n-t\choose k-t}$ maximizing total intersections?
\end{question}

In this paper, we show that $\mathcal{I}(\mathcal{F}_0)=\mathcal{MI}(\mathcal{F}_0)$ and full $t$-star is indeed the only structure of the extremal family, which answers Question \ref{question1}. Actually, we consider the following more general question:
\begin{question}\label{question}
For a family of subsets $\mathcal{F}\subseteq {[n]\choose k}$, if $\mathcal{I}(\mathcal{F})=\mathcal{MI}(\mathcal{F})$, what can we say about its structure?
\end{question}
Noticed that families of size larger than ${n-1\choose k-1}$ are no longer intersecting, therefore, this question concerns the family with a general intersecting property beyond the setting of Erd\H{o}s-Ko-Rado type problems.

Aiming to solve these questions, we provide two structural characterizations of the optimal family satisfying $\mathcal{I}(\mathcal{F})=\mathcal{MI}(\mathcal{F})$ for several ranges of size of $\mathcal{F}$. These results show that the optimal family is contained between two families prior in the lexicographic ordering with size of different levels. As a consequence, for $t\geq 1$ and $n$ large enough, the optimal family $\mathcal{F}$ is indeed $t$-intersecting when $|\mathcal{F}|\leq {n-t\choose k-t}$ is not too small. Also, these characterizations determine the unique structure of the optimal family and the exact value of $\mathcal{MI}(\mathcal{F})$ for several cases. Moreover, this also leads to an upper bound on $\mathcal{MI}(\mathcal{F})$ for these ranges of $|\mathcal{F}|$. The detailed description of our results will be shown in the following subsection.

\subsection{Our results}

For $F_1,F_2\in {n\choose k}$, denote $F_1\Delta F_2=(F_1\setminus F_2)\cup(F_2\setminus F_1)$ as the symmetric difference of $F_1$ and $F_2$. We say $F_1$ succeeds $F_2$ under the lexicographic ordering if the minimal element of $F_1\Delta F_2$ is in $F_1$, and we write $F_1\leq_{lex}F_2$. Given a positive integer $M$, denote $\mathcal{L}_{n,k}(M)$ as the first $M$ $k$-subsets of $[n]$ under the lexicographic ordering. Particularly, for $t\geq 1$, denote $\mathcal{L}_{n,k,t}^{(r)}$ as the first ${{n-t+1}\choose {k-t+1}}-{{n-(t+r-1)}\choose {k-t+1}}$ $k$-subsets of $[n]$ under the lexicographic ordering. Given $k\geq 2$, $r\geq 0$ and $t\geq 2$, for $1\leq s\leq t$, let $C_s=2^{2^{s-1}-1}\cdot 10^{2^{s+2}-2}\cdot (k^2t^4(r+1)^7)^{2^{s-1}}$ be a constant unrelated with $n$. Our main results are as follows.

\begin{theorem}\label{main0_t=1}
Let $C_0\geq 3\times 10^3$ be an absolute constant and $k\geq 2$, $r\geq 0$ be two fixed integers. For any $n\geq C_0(r+1)^3(k+r)k^{2}$ and $\delta\in[\frac{150k^3(r+1)^2}{n},1-\frac{150k^3(r+1)^3}{n}]\cup\{1\}$, if $\mathcal{F}\subseteq{[n]\choose k}$ with $|\mathcal{F}|=\sum_{i=1}^{r}{{n-i}\choose {k-1}}+\delta{{n-(r+1)}\choose {k-1}}$ and satisfies $\mathcal{I}(\mathcal{F})=\mathcal{MI}(\mathcal{F})$, then
\begin{equation*}
\mathcal{L}_{n,k,1}^{(r)}\subseteq \mathcal{F}\subseteq \mathcal{L}_{n,k,1}^{(r+1)},
\end{equation*}
up to isomorphism.
\end{theorem}
\begin{theorem}\label{main0}
Let $k\geq 2$, $r\geq 0$ and $t\geq 2$ be three fixed integers. For any $n\geq C_1\cdot(3tC_t)^{2t}$ and $\delta\in[\frac{60k^2(r+1)^6t^4}{C_1},1-\frac{60k^2(r+1)^6t^4}{C_1}]\cup\{1\}$, if $\mathcal{F}\subseteq{[n]\choose k}$ with $|\mathcal{F}|=\sum_{i=t}^{t+r-1}{{n-i}\choose {k-t}}+\delta{{n-(t+r)}\choose {k-t}}$ and satisfies $\mathcal{I}(\mathcal{F})=\mathcal{MI}(\mathcal{F})$. Then
\begin{equation*}
\mathcal{L}_{n,k,t}^{(r)}\subseteq \mathcal{F}\subseteq \mathcal{L}_{n,k,t}^{(r+1)},
\end{equation*}
up to isomorphism.
\end{theorem}

Denote $R_1=[\frac{150k^3(r+1)^2}{n},1-\frac{150k^3(r+1)^3}{n}]\cup\{1\}$ and $R_t=[\frac{60k^2(r+1)^6t^4}{C_1},1-\frac{60k^2(r+1)^6t^4}{C_1}]\cup\{1\}$, for $t\geq 2$. As a direct consequence of the above two theorems, families of proper sizes that maximize total intersections are indeed $t$-intersecting.
\begin{corollary}\label{coro0}
Let $k$, $r$, $t\geq1$ and $n$ be non-negative integers defined in Theorem \ref{main0}. If $|\mathcal{F}|=\delta{{n-t}\choose {k-t}}$ for some $\delta\in R_t$ satisfying $\mathcal{I}(\mathcal{F})=\mathcal{MI}(\mathcal{F})$. Then, $\mathcal{F}$ is a $t$-intersecting family.
\end{corollary}

Moreover, we have the following two corollaries of Theorem \ref{main0_t=1} that determines the unique structure of the optimal family for certain values of $|\mathcal{F}|$ and provides a general upper bound on $\mathcal{I}(\mathcal{F})$, respectively.
\begin{corollary}\label{coro1}
Let $k$, $r$ and $n$ be positive integers defined in Theorem \ref{main0_t=1}. If $|\mathcal{F}|=\sum_{i=1}^{r}{{n-i}\choose {k-1}}$ satisfying $\mathcal{I}(\mathcal{F})=\mathcal{MI}(\mathcal{F})$. Then, up to isomorphism, we have $\mathcal{F}=\mathcal{L}_{n,k,1}^{(r)}$.
\end{corollary}
\begin{corollary}\label{coro2}
Let $k$, $r$, $n$ and $\delta$ be the same as those defined in Theorem \ref{main0_t=1}. For any family $\mathcal{F}\subseteq{[n]\choose k}$ of size $\sum_{i=1}^{r}{{n-i}\choose {k-1}}+\delta{{n-(r+1)}\choose {k-1}}$, we have $\mathcal{I}(\mathcal{F})\leq (r+\delta^2){n-1\choose k-1}^2+(n-r-\lfloor\delta\rfloor)(\sum_{i=2}^{r+1}{{n-i}\choose k-2})^{2}$.
\end{corollary}

Also, we have similar corresponding corollaries of Theorem \ref{main0}.

\begin{corollary}\label{coro3}
Let $k$, $r$, $t\ge2$ and $n$ be positive integers defined in Theorem \ref{main0}. If $|\mathcal{F}|=\sum_{i=t}^{t+r-1}{{n-i}\choose {k-t}}$ satisfying $\mathcal{I}(\mathcal{F})=\mathcal{MI}(\mathcal{F})$. Then, up to isomorphism, we have $\mathcal{F}=\mathcal{L}_{n,k,t}^{(r)}$.
\end{corollary}
\begin{corollary}\label{coro4}
Let $k$, $r$, $t\ge2$, $n$ and $\delta$ be the same as those defined in Theorem \ref{main0}. For any family $\mathcal{F}\subseteq{[n]\choose k}$ of size $\sum_{i=t}^{t+r-1}{{n-i}\choose {k-t}}+\delta{{n-(t+r)}\choose {k-t}}$, we have $\mathcal{I}(\mathcal{F})\leq (t-1)|\mathcal{F}|^2+(r+\delta^2){n-t\choose k-t}^2+(n-(t+r+\lfloor\delta\rfloor-1))(\sum_{i=t+1}^{t+r}{{n-i}\choose k-(t+1)})^{2}$.
\end{corollary}

\subsection{Outline and notations}

We use the following standard mathematical notations throughout this paper.

\begin{itemize}
  \item Denote $\mathbb{N}$ as the set of all non-negative integers. For any $n\in\mathbb{N}\setminus\{0\}$, let $[n]=\{1,2,\cdots,n\}$. For any $a,b\in\mathbb{N}$ such that $a\leq b$, let $[a,b]=\{a,a+1,\cdots,b\}$.
  \item For given finite set $S\subseteq \mathbb{N}$ and any positive integer $k$, denote ${S\choose k}$ as the family of all $k$-subsets of $S$.
  \item For a given family $\mathcal{F}$ in ${[n]\choose k}$ and a $t$-subset $A\subseteq[n]$, we denote $\mathcal{F}(A)=\{F\in \mathcal{F}: A\in F\}$ as the subfamily of $\mathcal{F}$ containing $A$ and call $\deg_{\mathcal{F}}(A)=|\mathcal{F}(A)|$ the degree of $A$ in $\mathcal{F}$. Moreover, when $t=1$ and $A=\{x\}$, we denote $\mathcal{F}(x)=\mathcal{F}(\{x\})$ for short.
  \item For a given family $\mathcal{F}\subseteq {[n]\choose k}$ and an integer $s>0$, a subset $U\subseteq [n]$ is called an \emph{$s$-cover} of $\mathcal{F}$ with size $|U|$, if for every $F\in\mathcal{F}$, $|F\cap U|\geq s$.
  \item For a given family $\mathcal{F}$ in ${[n]\choose k}$ and $A\in\mathcal{F}$, the \emph{shifting operator} $\mathcal{S}_{i,j}$ is defined as follows:
      \begin{equation}
      \mathcal{S}_{i,j}(A)=\begin{cases}
      A\setminus \{i\}\cup\{j\}, \text{ if } i\in A, j\notin A \text{ and } A\setminus \{i\}\cup\{j\}\notin \mathcal{F};\\
      A, \text{~otherwise}.
      \end{cases}
      \end{equation}
      And we define $\mathcal{S}_{i,j}(\mathcal{F})=\{\mathcal{S}_{i,j}(A):A\in \mathcal{F}\}$.
\end{itemize}

The remainder of the paper is organized as follows. In Section 2, we consider families of size $O({n-1\choose k-1})$ and prove Theorem \ref{main0_t=1}.  In Section 3, we consider families of size $O({n-t\choose k-t})$ for $t\ge 2$ and prove Theorem \ref{main0}. In Section 4, we will conclude the paper and discuss some remaining problems.

\section{Proof of Theorem \ref{main0_t=1}}
In this section, we present the proof of Theorem \ref{main0_t=1}. The main tool that we use in this proof is the \emph{quantitative shifting method} introduced in \cite{DGS2016}. To carry out this method, our proof is divided into the following three steps:
\begin{itemize}
  \item First, to guarantee its optimality, we shall prove that the family $\mathcal{F}$ must contain a popular element, i.e., there is some $x\in [n]$ in many sets of $\mathcal{F}$. Based on this argument, we can prove the result when $\mathcal{F}$ contains a full $1$-star by induction.
  \item Second, when $\mathcal{F}$ does not contain any full $1$-star, we can replace the $k$-sets in $\mathcal{F}$ consisting of less popular elements with new $k$-sets containing this popular element. Through an estimation about the increment of $\mathcal{I}(\mathcal{F})$, we will show that $\mathcal{F}$ can be covered by $r+1$ elements in $[n]$.
  \item Finally, based on the results from former steps, we complete the proof by induction on $n$ and $r$.
\end{itemize}


\begin{lemma}\label{mainl01}
Let $k\geq 2$, $r$ and $n$ be non-negative integers defined in Theorem \ref{main0_t=1}. If $\mathcal{F}\subseteq{[n]\choose k}$ with size
\begin{equation*}
|\mathcal{F}|=\sum_{i=1}^{r}{{n-i}\choose {k-1}}+\delta{{n-(r+1)}\choose {k-1}}
\end{equation*}
for some $\delta\in[\frac{150k^3}{n},1]$, and satisfies $\mathcal{I}(\mathcal{F})\geq\frac{r+\delta^2}{(r+\delta)^2}|\mathcal{F}|^2$. Then, there exists an $x\in[n]$ with $|\mathcal{F}(x)|\geq \frac{|\mathcal{F}|}{4(r+1)}$.
\end{lemma}

\begin{proof}
First, take $X=\{x\in[n]:|\mathcal{F}(x)|\geq \frac{|\mathcal{F}|}{5k(r+1)}\}$ as the set of moderately popular elements, we show that $X$ can not be very large.

\textbf{Claim 1.} $|X|< 10k(r+1)$.

\begin{proof}

Suppose not, let $X_0$ be a subset of $X$ with size $10k(r+1)$, then we have
\begin{align}\label{ineq01}
|\mathcal{F}|\geq |\bigcup_{x\in X_0}\mathcal{F}(x)|&\geq \sum_{x\in X_0}|\mathcal{F}(x)|-\sum_{x\neq y\in X_0}|\mathcal{F}(x,y)|\\
&\geq 2|\mathcal{F}|-{|X_0|\choose 2}{{n-2}\choose {k-2}}.\nonumber
\end{align}
Since $|\mathcal{F}|=|\mathcal{L}_{n,k,1}^{(r)}|+\delta{n-(r+1)\choose k-1}$, by the choice of $n$ and Bonferroni Inequalities, we know that
\begin{align}\label{ineq02}
|\mathcal{F}|&\geq {r\choose 1}{{n-1}\choose {k-1}}-{r\choose 2}{n-2\choose k-2}+\delta{n-(r+1)\choose k-1}\\
&\geq (\frac{nr}{3k}+\delta\frac{n-(r+k)}{k-1}(1-\frac{k(r+k)}{n-2})){n-2\choose k-2}.\nonumber
\end{align}

Combining (\ref{ineq01}) and (\ref{ineq02}) together, we have
\begin{equation*}
|\mathcal{F}|> (2-\frac{150k^3(r+1)^2}{n(r+\delta)})|\mathcal{F}|,
\end{equation*}
which contradicts the requirement of $n$. Therefore, the claim holds.
\end{proof}

Now, we complete the proof by proving the following claim.

\textbf{Claim 2.} There is an $x_0\in X$ such that $|\mathcal{F}(x_0)|\geq \frac{|\mathcal{F}|}{4(r+1)}$.

\begin{proof}


W.l.o.g., assume that $1\in X$ is the most popular element appearing in $\mathcal{F}$. Then, we have
\begin{align}\label{ineq04}
\mathcal{I}(\mathcal{F})&=\sum_{x\in X}|\mathcal{F}(x)|^2+\sum_{x\in [n]\setminus X}|{\mathcal{F}(x)}|^{2} \nonumber\\
&\leq |\mathcal{F}(1)|\cdot\sum_{x\in X}|\mathcal{F}(x)|+\frac{|\mathcal{F}|}{5k(r+1)}\cdot\sum_{x\in [n]\setminus X}|\mathcal{F}(x)| \nonumber\\
&\leq |\mathcal{F}(1)|\cdot(|\mathcal{F}|+{|X|\choose 2}{n-2\choose k-2})+\frac{|\mathcal{F}|}{5k(r+1)}\cdot k\cdot|\mathcal{F}| \nonumber\\
&\leq |\mathcal{F}(1)|\cdot|\mathcal{F}|\cdot(1+\frac{150k^3(r+1)^2}{n(r+\delta)})+\frac{|\mathcal{F}|^2}{5(r+1)}.
\end{align}

By the lower bound of $\mathcal{I}(\mathcal{F})$ and $(\ref{ineq04})$, we can obtain
\begin{equation*}
|\mathcal{F}(1)|\geq \frac{3|\mathcal{F}|}{10(r+1)\cdot (1+\frac{150k^3(r+1)^2}{n(r+\delta)})}\geq \frac{|\mathcal{F}|}{4(r+1)}.
\end{equation*}
Therefore, the claim holds.
\end{proof}
This completes the proof.
\end{proof}

Given a subset $A\subseteq [n]$ and a family of subsets $\mathcal{F}\subseteq {[n]\choose k}$, we say $A$ is a \emph{cover} of $\mathcal{F}$ if for every $F\in \mathcal{F}$, $A\cap F\neq \emptyset$. Based on Lemma \ref{mainl01}, we can proceed to the second step.

\begin{lemma}\label{mainl02}
Let $\mathcal{F}\subseteq{[n]\choose k}$ be the same family as that in Theorem \ref{main0_t=1}. If $\mathcal{F}$ does not contain any full $1$-star, then $\mathcal{F}$ has a cover $A\subseteq[n]$ of size $r+1$.
\end{lemma}

\begin{proof}
First, we show that the set of moderately popular elements already forms a cover of $\mathcal{F}$.

\textbf{Claim 3.} $X=\{x\in[n]:\mathcal{F}(x)\geq \frac{|\mathcal{F}|}{5k(r+1)}\}$ is a cover of $\mathcal{F}$.

\begin{proof}
Suppose not, there exists an $F_0\in \mathcal{F}$ such that $F_0\cap X=\emptyset$. Thus, for every $x\in F_0$, $\mathcal{F}(x)< \frac{|\mathcal{F}|}{5k(r+1)}$. Since
\begin{align}\label{ineq05}
\mathcal{I}(\mathcal{F})=\sum_{F\in \mathcal{F}}\mathcal{I}(F,\mathcal{F})=k+2\mathcal{I}(F_0,\mathcal{F}\setminus\{F_0\})+\mathcal{I}(\mathcal{F}\setminus\{F_0\}).
\end{align}
Noted that the unpopularity of the elements in $F_0$ may lead to $\mathcal{I}(F_0,\mathcal{F}\setminus\{F_0\})$ being very small, thus, if it is possible, we can increase the value of $\mathcal{I}(F_0,\mathcal{F}\setminus\{F_0\})$ by replacing $F_0$ with another $k$-subset of $[n]$ containing a popular element without changing the value of $\mathcal{I}(\mathcal{F}\setminus\{F_0\})$.

In fact, this is possible. Due to the assumption that $\mathcal{F}(1)\subsetneq\mathcal{L}_{n,k,1}^{(1)}$ (i.e., $\mathcal{F}$ contains no full $1$-star), we can replace $F_0$ with some $F_0'\in \mathcal{L}_{n,k,1}^{(1)}\setminus \mathcal{F}$. Denote the new family as $\mathcal{F}'$, we have
\begin{align*}
\mathcal{I}(\mathcal{F}')-\mathcal{I}(\mathcal{F})&=2(\mathcal{I}(F_0',\mathcal{F}'\setminus\{F_0'\})-\mathcal{I}(F_0,\mathcal{F}\setminus\{F_0\}))+\mathcal{I}(\mathcal{F}'\setminus\{F_0'\})-\mathcal{I}(\mathcal{F}\setminus\{F_0\}) \\
&=2(\mathcal{I}(F_0',\mathcal{F}\setminus\{F_0\})-\mathcal{I}(F_0,\mathcal{F}\setminus\{F_0\})) \\
&\geq 2(\sum_{x\in F_0'}|\mathcal{F}(x)|-\sum_{x\in F_0}|\mathcal{F}(x)|) \\
&\geq 2(|\mathcal{F}(1)|-\frac{|\mathcal{F}|}{5(r+1)})>0,
\end{align*}
which contradicts the optimality of $\mathcal{F}$. Therefore, the claim holds.
\end{proof}

By Claim 3, we know that $\mathcal{F}$ has a cover $X$ with size less than $10k(r+1)$. Let $X_0\subseteq X$ be the minimal cover of $\mathcal{F}$ containing $1$. W.l.o.g., assume that $X_0=[m]$. For each $i\in [m]$, denote $\mathcal{F}^{*}(i)$ as the subfamily in $\mathcal{F}(i)$ consisting of all $k$-sets with $i$ as their minimal element. Then, $\mathcal{F}=\bigsqcup_{i=1}^{m}\mathcal{F}^{*}(i)$. Thus, we have the following claim.

\textbf{Claim 4.} For every $i,j\in [m]$, we have $|\mathcal{F}^{*}(i)|\geq |\mathcal{F}^{*}(j)|-\frac{3mk^2}{(r+\delta)n}|\mathcal{F}|$.

\begin{proof}
First we claim that for each $i\in [m]$, there is some $F\in \mathcal{F}(i)$ such that $F\cap [m]=\{i\}$. Otherwise, suppose that for every $F\in \mathcal{F}(i)$, we have $|F\cap[m]|\geq 2$. Then,
\begin{equation*}
|\mathcal{F}(i)|\leq (m-1){n-2\choose k-2}< \frac{3mk}{(r+\delta)n}|\mathcal{F}|<\frac{|\mathcal{F}|}{5k(r+1)}.
\end{equation*}
This contradicts the fact that $\mathcal{F}(i)\geq \frac{|\mathcal{F}|}{5k(r+1)}$. 

Now, assume there exist $i_0\neq j_0\in[m]$ satisfying $|\mathcal{F}^{*}(j_0)|>|\mathcal{F}^{*}(i_0)|+\frac{3mk^2}{(r+\delta)n}|\mathcal{F}|$. Thus, since $1$ is the most popular element in $[n]$ and $\mathcal{F}^{*}(1)=\mathcal{F}(1)$, we know that $i_0\neq 1$ and
\begin{equation*}
|\mathcal{F}^{*}(1)|\geq|\mathcal{F}^{*}(j_0)|>|\mathcal{F}^{*}(i_0)|+\frac{3mk^2}{(r+\delta)n}|\mathcal{F}|.
\end{equation*}
Noted that $\mathcal{F}^{*}(1)\subsetneq \mathcal{L}_{n,k,1}^{(1)}$, therefore, we can replace the $k$-subset $F\in\mathcal{F}^{*}(i_0)$ satisfying $F\cap [m]=\{i_0\}$ with some $F'\in\mathcal{L}_{n,k,1}^{(1)}\setminus\mathcal{F}^{*}(1)$. Let $\mathcal{F}'$ be the resulting new family, by (\ref{ineq05}), we have
\begin{align*}
\mathcal{I}(\mathcal{F}')-\mathcal{I}(\mathcal{F}) &=2(\mathcal{I}(F',\mathcal{F}'\setminus\{F'\})-\mathcal{I}(F,\mathcal{F}\setminus\{F\}))\\
&\geq 2(\mathcal{I}(F',\mathcal{F}^{*}(1))-\sum_{x\in F}\sum_{i\in [m]}|\{A\in \mathcal{F}^{*}(i):x\in A\}|) \\
&\geq 2(|\mathcal{F}^{*}(1)|-|\mathcal{F}^{*}(i_0)|-k(m-1){n-2\choose k-2})>0.
\end{align*}
This contradicts the optimality of $\mathcal{F}$. Therefore, the claim holds.
\end{proof}

Actually, Claim 4 shows that as the extremal family, the sizes of sub-families $\mathcal{F}^{*}(i)$ ($i\in [m]$) of $\mathcal{F}$ are relatively close. Since $|\mathcal{F}^{*}(1)|=|\mathcal{F}(1)|\geq \frac{|\mathcal{F}|}{4(r+1)}$, thus for each $i\neq 1\in [m]$,
\begin{equation*}
|\mathcal{F}^{*}(i)|\geq |\mathcal{F}^{*}(1)|-\frac{3mk^2}{(r+\delta)n}|\mathcal{F}|\geq\frac{|\mathcal{F}|}{20(r+1)}.
\end{equation*}
Noticed that $\{\mathcal{F}^{*}(i)\}_{i=1}^{m}$ forms a partition of $\mathcal{F}$, this leads to a rough bound on $m$ as: $m\leq 20(r+1)$.

Based on this rough bound, we complete the proof by proving the next claim.

\textbf{Claim 5.} $m=r+1$.
\begin{proof}
We only prove the case when $r >0$, for $r=0$ the proof is the same.

Given two $k$-uniform families $\mathcal{F}_1$ and $\mathcal{F}_2$, we define $\mathcal{I}(\mathcal{F}_1,\mathcal{F}_2)=\sum_{A\in\mathcal{F}_1, B\in\mathcal{F}_2}|A\cap B|$. Clearly, we have $\mathcal{I}(\mathcal{F})=\mathcal{I}(\mathcal{F},\mathcal{F})$. By the size of $\mathcal{F}$, $m\geq r+1$. Assume that $m\geq r+2$. First, we have
\begin{align}\label{ineq05.5}
\mathcal{I}(\mathcal{F})=&\sum_{i,j\in [m]}\mathcal{I}(\mathcal{F}^{*}(i),\mathcal{F}^{*}(j))=\sum_{i\in [m]}\mathcal{I}(\mathcal{F}^{*}(i))+\sum_{i\neq j\in [m]}\mathcal{I}(\mathcal{F}^{*}(i),\mathcal{F}^{*}(j))\nonumber\\
=&\sum_{i\in [m]}\sum_{F\in \mathcal{F}^{*}(i)}\sum_{x\in F}|\{A\in \mathcal{F}^{*}(i): x\in A\}|+\sum_{i\neq j\in [m]}\sum_{F\in \mathcal{F}^{*}(i)}\sum_{x\in F}|\{A\in \mathcal{F}^{*}(j): x\in A\}| \nonumber\\
\leq&\sum_{i\in [m]}(|\mathcal{F}^{*}(i)|+(k-1){n-2\choose k-2})|\mathcal{F}^{*}(i)|+\sum_{i\neq j\in [m]}k{n-2\choose k-2}(|\mathcal{F}^{*}(i)|+|\mathcal{F}^{*}(j)|) \nonumber\\
\leq&\sum_{i\in [m]}|\mathcal{F}^{*}(i)|^{2}+2km{n-2\choose k-2}|\mathcal{F}|\leq (|\mathcal{F}^{*}(1)|+2km{n-2\choose k-2})|\mathcal{F}|.
\end{align}

By a simple averaging argument, there exists some $i_0\in [m]$ such that $|\mathcal{F}^{*}(i_0)|\leq \frac{|\mathcal{F}|}{m}$. Thus, by Claim 4, we have
\begin{equation*}
|\mathcal{F}^{*}(1)|\leq (\frac{1}{m}+\frac{3mk^2}{(r+\delta)n})|\mathcal{F}|.
\end{equation*}
This leads to
\begin{equation}\label{ineq06}
\mathcal{I}(\mathcal{F})\leq (\frac{1}{m}+\frac{9mk^2}{(r+\delta)n})|\mathcal{F}|^{2}\leq (\frac{1}{r+2}+\frac{360k^2}{n})|\mathcal{F}|^2.
\end{equation}

Since $\mathcal{F}$ is the extremal family, we know that
\begin{align*}
\mathcal{I}(\mathcal{F})\geq \mathcal{I}(\mathcal{L}_{n,k}(|\mathcal{F}|))&=\sum_{x\in [n]}|F\in \mathcal{L}_{n,k}(|\mathcal{F}|):x\in F|^{2} \nonumber\\
&\geq (r+\delta^2){n-1\choose k-1}^2\geq \frac{r+\delta^2}{(r+\delta)^2}|\mathcal{F}|^2.
\end{align*}
Combining with (\ref{ineq06}), we can obtain
\begin{equation*}
\frac{r+\delta^2}{(r+\delta)^2}|\mathcal{F}|^2\leq \mathcal{I}(\mathcal{F})\leq (\frac{1}{r+2}+\frac{360k^2}{n})|\mathcal{F}|^2.
\end{equation*}
Since $n\geq C_0(r+1)^3(k+r)k^2$, we have $\frac{360k^2}{n}<\frac{r+\delta^2}{(r+\delta)^2}-\frac{1}{r+2}$, a contradiction.

Therefore, $m= r+1$.
\end{proof}
This completes the proof.
\end{proof}

\begin{proof}[Proof of Theorem \ref{main0_t=1}]

We prove the theorem by induction on $n$ and $r$.

Consider the base case: $r=0$. By Lemma \ref{mainl02}, we know that $\mathcal{F}$ has a cover of size $1$. Noted that we have already assumed that $1\in [n]$ is the most popular element of $\mathcal{F}$, thus $\mathcal{F}=\mathcal{F}(1)$. This indicates that $\mathcal{L}_{n,k,1}^{(0)}\subseteq \mathcal{F}\subseteq \mathcal{L}_{n,k,1}^{(1)}$.

Now, suppose that $\mathcal{F}$ contains a full $1$-star. W.l.o.g., assume this full $1$-star consists of all $k$-sets containing $1$. Then, we have $r\geq 1$ and
\begin{align}\label{eq00}
\mathcal{I}(\mathcal{F})&=\mathcal{I}(\mathcal{F}(1))+\mathcal{I}(\mathcal{F}\setminus\mathcal{F}(1),\mathcal{F})\nonumber\\
&=({n-1\choose k-1}^{2}+(n-1){n-2\choose k-2}^{2})+\sum_{A\in \mathcal{F}\setminus\mathcal{F}(1)}\mathcal{I}(A,\mathcal{F}).
\end{align}
And for any $A\in \mathcal{F}\setminus\mathcal{F}(1)$,
\begin{align}\label{eq01}
\mathcal{I}(A,\mathcal{F})&=\mathcal{I}(A,\mathcal{F}(1))+\mathcal{I}(A,\mathcal{F}\setminus\mathcal{F}(1)) \nonumber \\
&=k{n-2\choose k-2}+\mathcal{I}(A,\mathcal{F}\setminus\mathcal{F}(1)).
\end{align}
Therefore,
\begin{equation*}
\mathcal{I}(\mathcal{F})=C_0(n,k)+\mathcal{I}(\mathcal{F}\setminus\mathcal{F}(1)),
\end{equation*}
where $C_0(n,k)=({n-1\choose k-1}^{2}+(n-1){n-2\choose k-2}^{2})+k{n-2\choose k-2}(|\mathcal{F}|-{n-1\choose k-1})$. Denote $\mathcal{F}'=\mathcal{F}\setminus\mathcal{F}(1)$, then $\mathcal{F}'$ can be viewed as a family of $k$-sets in ${[n]\setminus\{1\}\choose k}$. Due to the optimality of $\mathcal{F}$, we have
\begin{equation*}
\mathcal{I}(\mathcal{F}')=\max_{\mathcal{G}\subseteq {[n]\setminus\{1\}\choose k}, |\mathcal{G}|=|\mathcal{F}'|}=\mathcal{I}(\mathcal{G})=\mathcal{MI}_{[n]\setminus\{1\}}(\mathcal{F}').
\end{equation*}
Thus, by induction hypothesis, $\mathcal{F}'\subseteq {[n]\setminus\{1\}\choose k}$ satisfies that $\mathcal{L}_{n-1,k,1}^{(r-1)}\subseteq \mathcal{F}'\subseteq \mathcal{L}_{n-1,k,1}^{(r)}$. Joined with the full $1$-star $\mathcal{F}(1)$, we have $\mathcal{L}_{n,k,1}^{(r)}\subseteq \mathcal{F}\subseteq \mathcal{L}_{n,k,1}^{(r+1)}$ as claimed.

When $\mathcal{F}$ does not contain any full $1$-star, by Lemma \ref{mainl02}, we know that $\mathcal{F}$ can be covered by an $(r+1)$-subset of $[n]$. W.l.o.g, assume that this $(r+1)$-subset is $[r+1]$.

Let $\mathcal{A}=\{A\in {[n]\choose k}: A\cap [r+1]\neq \emptyset\}$ be the family of all $k$-subsets that intersect $[r+1]$. When $\delta=1$, we have $\mathcal{F}=\mathcal{A}=\mathcal{L}_{n,k,1}^{(r+1)}$. When $\delta\neq 1$, $\mathcal{F}\subsetneq \mathcal{A}$. Let $\mathcal{G}=\mathcal{A}\setminus \mathcal{F}$, we have
\begin{align}\label{ineq07}
\mathcal{I}(\mathcal{F})&=\mathcal{I}(\mathcal{A})-2\mathcal{I}(\mathcal{G},\mathcal{A})+\mathcal{I}(\mathcal{G}) \\
&=\mathcal{I}(\mathcal{A})-2\sum_{G\in \mathcal{G}}\mathcal{I}(G,\mathcal{A})+\mathcal{I}(\mathcal{G}). \nonumber
\end{align}
Note that once $r+1$ is given, $\mathcal{A}$ can be viewed as the union of $r+1$ full $1$-stars with cores $1,2,\ldots,r+1$. Based on this structure, for each $x\in [r+1]$, we have $\mathcal{A}(x)={n-1\choose k-1}$ and for each $x\in [n]\setminus [r+1]$, we have $\mathcal{A}(x)=\sum_{i=1}^{r+1}{n-i-1\choose k-2}$. Since $\mathcal{I}(\mathcal{A})=\sum_{x\in [n]}|\mathcal{A}(x)|^{2}$ and for each $G\in \mathcal{A}$, $\mathcal{I}(G,\mathcal{A})=\sum_{x\in G}|\mathcal{A}(x)|$, thus $\mathcal{I}(\mathcal{A})$ and $\mathcal{I}(G,\mathcal{A})$ are both fixed constants.

By (\ref{ineq07}), the optimality of $\mathcal{F}$ is actually guaranteed by $\mathcal{I}(\mathcal{G})-2\sum_{G\in \mathcal{G}}\mathcal{I}(G,\mathcal{A})$, i.e., $\mathcal{I}(\mathcal{F})=\mathcal{MI}(\mathcal{F})$ if and only if $\mathcal{I}(\mathcal{G})-2\sum_{G\in \mathcal{G}}\mathcal{I}(G,\mathcal{A})$ reaches the maximum. Based on this observation, we have the following claim.

\textbf{Claim 6.} For $\mathcal{G}\subseteq \mathcal{A}$ with size $|\mathcal{G}|=|\mathcal{A}|-|\mathcal{F}|$, $\mathcal{I}(\mathcal{G})-2\sum_{G\in \mathcal{G}}\mathcal{I}(G,\mathcal{A})$ reaches its maximum only if there exists some $i_0\in [r+1]$ such that $G\cap [r+1]=\{i_0\}$ for all $G\in \mathcal{G}$.

\begin{proof}
First, we show that for all $G\in \mathcal{G}$, $|G\cap [r+1]|=1$. Otherwise, assume that there exists $G_0\in\mathcal{G}$ satisfying $|G_0\cap[r+1]|\geq 2$. W.l.o.g., assume that $1$ is the most popular element in $\mathcal{G}$ among $[r+1]$. Since $\mathcal{G}$ contains no full $1$-star, we can replace $G_0$ with some $G_1\in \mathcal{A}(1)\setminus\mathcal{G}(1)$. Denote the resulting new family as $\mathcal{G}'$, we have
\begin{align*}
[\mathcal{I}(\mathcal{G}')-2\sum_{G\in \mathcal{G}'}\mathcal{I}(G,\mathcal{A})]-[\mathcal{I}(\mathcal{G})-2\sum_{G\in \mathcal{G}}\mathcal{I}(G,\mathcal{A})]=(\mathcal{I}(\mathcal{G}')-\mathcal{I}(\mathcal{G}))+2(\sum_{G\in \mathcal{G}}\mathcal{I}(G,\mathcal{A})-\sum_{G\in \mathcal{G}'}\mathcal{I}(G,\mathcal{A})).
\end{align*}
From (\ref{ineq05}), we know that
\begin{align*}
\mathcal{I}(\mathcal{G}')-\mathcal{I}(\mathcal{G})&=2(\mathcal{I}(G_1,\mathcal{G}\setminus\{G_0\})-\mathcal{I}(G_0,\mathcal{G}\setminus\{G_0\}))\\
&\geq 2(\sum_{x\in G_1}|\mathcal{G}(x)|-\sum_{x\in G_0}|\mathcal{G}(x)|)\\
&\geq -2|\mathcal{G}|-r(r+1){n-2\choose k-2}.
\end{align*}
On the other hand, we have
\begin{align*}
\sum_{G\in \mathcal{G}}\mathcal{I}(G,\mathcal{A})-\sum_{G\in \mathcal{G}'}\mathcal{I}(G,\mathcal{A})&=\mathcal{I}(G_0,\mathcal{A})-\mathcal{I}(G_1,\mathcal{A})\\
&\geq {n-1\choose k-1}-(k-1)(r+1){n-2\choose k-2}.
\end{align*}
By combining these two estimations together, we have
\begin{align*}
[\mathcal{I}(\mathcal{G}')-2\sum_{G\in \mathcal{G}'}\mathcal{I}(G,\mathcal{A})]-[\mathcal{I}(\mathcal{G})-2\sum_{G\in \mathcal{G}}\mathcal{I}(G,\mathcal{A})]\geq 2\delta {n-1\choose k-1}-(2k+r)(r+1){n-2\choose k-2}>0,
\end{align*}
where the first inequality follows from $|\mathcal{G}|=(1-\delta){n-(r+1)\choose k-1}$ and the second inequality follows from the choice of $\delta$. This contradicts the maximality of $\mathcal{I}(\mathcal{G})-2\sum_{G\in \mathcal{G}}\mathcal{I}(G,\mathcal{A})$.

Noticed that $|G\cap [r+1]|=1$ for all $G\in \mathcal{G}$ indicates that $\mathcal{I}(G,\mathcal{A})={n-1\choose k-1}+(k-1)\sum_{i=1}^{r+1}{n-i-1\choose k-2}$, which is a constant irrelevant to the structure of $\mathcal{G}$. Therefore, $\mathcal{I}(\mathcal{G})-2\sum_{G\in \mathcal{G}}\mathcal{I}(G,\mathcal{A})$ attains its maximum only if
\begin{equation*}
\mathcal{I}(\mathcal{G})=\max_{\substack{\mathcal{G}_0\subseteq \mathcal{A}, |\mathcal{G}_0|=|\mathcal{G}|\\ |G\cap [r+1]|=1 \text{~for all~} G\in \mathcal{G}_0}}\mathcal{I}(\mathcal{G}_0).
\end{equation*}
Noted that for $\mathcal{G}_0\subseteq \mathcal{A}(1)$ with $|\mathcal{G}_0|=|\mathcal{G}|$, $\mathcal{I}(\mathcal{G}_0)\geq |\mathcal{G}|^2$, thus we have $\mathcal{I}(\mathcal{G})\geq |\mathcal{G}|^2$. Since $\mathcal{G}=\bigsqcup_{i=1}^{r+1}\mathcal{G}(i)$, by the upper bound of $\mathcal{I}(\mathcal{G})$ in (\ref{ineq05.5}), we have $|\mathcal{G}(1)|\geq |\mathcal{G}|-2k(r+1){n-2\choose k-2}$.
Therefore, through a similar shifting argument as Claim 4, the maximality of $\mathcal{I}(\mathcal{G})$ guarantees that $|\mathcal{G}(i)|\geq |\mathcal{G}(1)|-\frac{6(r+1)k^2}{(1-\delta)n}|\mathcal{G}|$ for every $i\in [r+1]$ with $|\mathcal{G}(i)|>0$. If there exists some $2\leq i'\leq r+1$ such that $|\mathcal{G}(i')|>0$, we shall have $\sum_{i\in [r+1]}|\mathcal{G}(i)|\geq (2-\frac{6(r+1)k^2}{(1-\delta)n})|\mathcal{G}|-2k(r+1){n-2\choose k-2}$, which contradicts the fact that $|\mathcal{G}|\geq \sum_{i\in[r+1]}|\mathcal{G}(i)|-{r+1\choose 2}{n-2\choose k-2}$. Therefore, we have $\mathcal{G}=\mathcal{G}(1)$ and the claim holds.
\end{proof}

By Claim 6, we know that $\mathcal{G}$ is contained in a full $1$-star of $\mathcal{A}$. W.l.o.g, assume that $\mathcal{G}\subseteq \{A\in {[n]\choose k}: r+1\in A\}$ and this leads to $\mathcal{L}_{n,k,1}^{(r)}\subseteq \mathcal{F}\subseteq \mathcal{L}_{n,k,1}^{(r+1)}$. This completes the proof of Theorem \ref{main0_t=1}.
\end{proof}

\begin{Remark}\label{rmk01}
According to the proof, one may wonder if the range that $\delta\in[\frac{150k^3(r+1)^2}{n}, 1-\frac{150k^3(r+1)^3}{n}]\cup\{1\}$ can be extended. Actually, the range might be improved to be a little bit larger, but anyway, $\delta$ can not be too close to $0$ or too close to $1$ when $\delta<1$.

For example, fix $k\geq 2$ and $n$ sufficiently large. Consider a family $\mathcal{F}_0\subseteq {[n]\choose k}$ with size $k+1$, one can easily verify that $\mathcal{I}(\mathcal{G})$ achieves the maximality when $\mathcal{F}_0={[k+1]\choose k}$. Clearly, for this case, $\mathcal{F}_0\nsubseteq \mathcal{L}_{n,k,1}^{(1)}$.
\end{Remark}

\section{Proof of Theorem \ref{main0}}

Recall the proof of Theorem \ref{main0_t=1}. First, we showed that $\mathcal{F}$ must contain a popular element to guarantee its optimality. Then, we showed that if $\mathcal{F}$ doesn't have a small cover, $\mathcal{I}({\mathcal{F}})$ can be increased through shifting arguments. This indicates that $\mathcal{F}$ must have a certain clustering property and can be covered by a few popular elements. Finally, noted that $|\mathcal{F}|$ is fixed, this small cover ensures $\mathcal{F}$ to have the desired structure.

For Theorem \ref{main0}, since the family we shall deal with is much sparser when $t\geq 2$, it requires more delicate analysis of the family $\mathcal{F}$ to proceed the above arguments. In order to prove Theorem \ref{main0}, we shall require a few preliminary results.


First, we need the following lemma from \cite{DGS2016} which shows that among all unions of $r$ full $t$-stars, the lexicographic ordering contains the fewest sets.

\begin{lemma}\label{mainl06}\cite{DGS2016}
Suppose $k\geq 2$, $t\geq 1$, $r$ and $n$ are given non-negative integers defined in Theorem \ref{main0}. Let $\mathcal{F}$ be the union of $r$ full $t$-stars in ${[n]\choose k}$. Then $|\mathcal{F}|\geq \sum_{i=t}^{t+r-1}{n-i\choose k-t}$, with equality to hold if and only if $\mathcal{F}$ is isomorphic to $\mathcal{L}_{n,k,t}^{(r)}$.
\end{lemma}

With the help of Lemma \ref{mainl06}, when $\mathcal{F}$ contains $p$ full $t$-stars and the total intersection of the remaining $k$-sets is well bounded, we have the following lemma which determines the structure of these $p$ full $t$-stars and shows that the remaining family is almost cross $(t-1)$-intersecting with each of these $p$ full $t$-stars.

\begin{lemma}\label{mainl07}
Let $k\geq 2$, $t\geq 1$, $r\geq 1$ and $n$ be given non-negative integers defined in Theorem \ref{main0}. Let $\mathcal{F}\subseteq{[n]\choose k}$ with size $|\mathcal{F}|=\sum_{i=t}^{t+r-1}{{n-i}\choose {k-t}}+\delta_0{{n-(r+t)}\choose {k-t}}$
for some $\delta_0\in[\frac{6k(r+1)t}{C_1},1]$, and satisfy $\mathcal{I}(\mathcal{F})=\mathcal{MI}(\mathcal{F})$. Suppose $\mathcal{F}$ contains $p$ full $t$-stars $\mathcal{Y}_1,\dots,\mathcal{Y}_p$ for some integer $1\leq p\leq r$ and
\begin{equation*}
\mathcal{I}(\mathcal{F}_0)\leq(t-1)|\mathcal{F}_0|^2+(r-p+\delta_0^2+\frac{1}{C_t}){n-t\choose k-t}^2,
\end{equation*}
for any $\mathcal{F}_0\subseteq \mathcal{F}$ with size $\sum_{i=t}^{t+r-(p+1)}{{n-i}\choose {k-t}}+\frac{r\delta_0}{r+1}{{n-(r+t-p)}\choose {k-t}}< |\mathcal{F}_0|\leq \sum_{i=t}^{t+r-(p+1)}{{n-i}\choose {k-t}}+\delta_0{{n-(r+t-p)}\choose {k-t}}$. Let $\mathcal{F}_1=\cup_{i=1}^{p}\mathcal{Y}_i$ and $\mathcal{F}_2=\mathcal{F}\setminus \mathcal{F}_1$. For each $1\leq i\leq p$, denote $Y_i\in {[n]\choose t}$ as the core of $\mathcal{Y}_i$, then
\begin{itemize}
  \item for all $i\neq j\in [p]$, $|Y_i\cap Y_j|=t-1$;
  \item for at least $(1-\frac{2r^2kt}{C_t})|\mathcal{F}_2|$ $k$-sets $F\in\mathcal{F}_2$, $|F\cap Y_i|=t-1$ for each $1\leq i\leq p$;
  \item $\mathcal{I}(\mathcal{F}_2)\geq (t-1)|\mathcal{F}_2|^2+(r-p+\delta_0^2-\frac{4kr^2p^2}{C_t}){n-t\choose k-t}^2$.
\end{itemize}
\end{lemma}

Let $\mathcal{F}_1=\cup_{i=t}^{t+r}\mathcal{G}_j$ and $\mathcal{F}_2=\cup_{j=1}^{t+1}\mathcal{H}_j$ with the same size $\sum_{i=t}^{t+r-1}{{n-i}\choose {k-t}}+\delta{{n-(r+t)}\choose {k-t}}$ for some $\delta\in[\frac{6krt}{C_1},1-\frac{6krt}{C_1}]\cup\{1\}$, where $\mathcal{G}_i$ is a $t$-star with core $\{1,\ldots,t-1,i\}$ and $\mathcal{H}_j$ is a $t$-star with core $[t+1]\setminus\{j\}$. The following lemma shows that when the size of each star is not too small, family with the structure of $\mathcal{F}_1$ has larger total intersection.

\begin{lemma}\label{mainl08}
Let $k\geq 2$, $t\geq 1$, $1\leq r\leq t$ and $n$ be given non-negative integers defined in Theorem \ref{main0}. Let $\mathcal{F}=\cup_{j=1}^{t+1}\mathcal{H}_j$ with $|\mathcal{F}|=\sum_{i=t}^{t+r-1}{{n-i}\choose {k-t}}+\delta{{n-(r+t)}\choose {k-t}}$ for some $\delta\in[\frac{6krt}{C_1},1-\frac{6krt}{C_1}]\cup\{1\}$, where $\mathcal{H}_j$ is a $t$-star with core $[t+1]\setminus\{j\}$. Assume that $|\mathcal{H}_j|\geq\frac{\delta}{2C_1}{n-t\choose k-t}$ for each $j\in[t+1]$. Then, there exists a family $\mathcal{F}_0$ with size $|\mathcal{F}|$ such that $\mathcal{L}_{n,k,t}^{(r)}\subseteq \mathcal{F}_0\subseteq \mathcal{L}_{n,k,t}^{(r+1)}$ and $\mathcal{I}(\mathcal{F}_0)>\mathcal{I}(\mathcal{F})$.
\end{lemma}

Our proof of Theorem \ref{main0} will proceed according to the following steps:

\begin{itemize}
  \item First, we show that if $\mathcal{I}(\mathcal{F})$ is large enough, $\mathcal{F}$ must have a popular $t$-set.
\end{itemize}
\begin{lemma}\label{mainl03}
Let $k\geq 2$, $t\geq 1$, $r$ and $n$ be given non-negative integers defined in Theorem \ref{main0}. If $\mathcal{F}\subseteq{[n]\choose k}$ with size
\begin{equation*}
|\mathcal{F}|=\sum_{i=t}^{t+r-1}{{n-i}\choose {k-t}}+\delta{{n-(r+t)}\choose {k-t}}
\end{equation*}
for some $\delta\in[\frac{6k(r+1)t}{C_1},1]$, and satisfies $\mathcal{I}(\mathcal{F})\geq (t-1)|\mathcal{F}|^{2}+(r+\delta^2-\epsilon_0){n-t\choose k-t}^2$ for some constant $\epsilon_0\leq \frac{\delta^2}{10(r+1)}$. Then, there exists some $A\in{[n]\choose t}$ with $|\mathcal{F}(A)|\geq \frac{r+\delta^2}{2t(r+\delta)^{2}}|\mathcal{F}|$. Moreover, when $r=0$, we have $|\mathcal{F}(A_0)|\geq  |\mathcal{F}|(1-\frac{2k^t}{C_t})$, where $A_0\in {[n]\choose t}$ is the most popular $t$-set in $\mathcal{F}$.
\end{lemma}

\begin{itemize}
  \item Second, we show that if $\mathcal{F}$ contains at most one full $t$-star, then $\mathcal{F}$ has a small $t$-cover. Moreover, if $\mathcal{F}$ contains no full $t$-star, all $F\in \mathcal{F}$ share a common element.
\end{itemize}
\begin{lemma}\label{mainl04}
Let $\mathcal{F}\subseteq{[n]\choose k}$ be the same family defined in Theorem \ref{main0}. For $t\geq 2$, if $\mathcal{F}$ contains at most one full $t$-star, then there exists a subset $U_t\subseteq [n]$ with $|U_t|\leq t(4r+5)$ being a $t$-cover of $\mathcal{F}$.
\end{lemma}
\begin{lemma}\label{mainl05}
Let $\mathcal{F}\subseteq{[n]\choose k}$ be the same family defined in Theorem \ref{main0}. If $\mathcal{F}$ contains no full $t$-star, then there exists an $x_0\in [n]$ such that $x_0\in F$, for all $F\in\mathcal{F}$.
\end{lemma}

\begin{itemize}
  \item Third, with the help of Lemmas \ref{mainl07} and \ref{mainl08}, by induction on $r$, we show that for the extremal family $\mathcal{F}$, all $F\in \mathcal{F}$ share a common element. This enables us to proceed the induction on $n$, $k$ and $t$ and therefore, Theorem \ref{main0} shall follow from Theorem \ref{main0_t=1}.
\end{itemize}

Since the estimation of $\mathcal{MI}(\mathcal{F})$ requires using the property of a certain convex function, we have the following theorem which is crucial during the proof of Theorem \ref{main0} and related lemmas.

\begin{theorem}\label{convex_property}
Given integer $r\geq 0$, let $f:\mathbb{R}^{r+1}\rightarrow\mathbb{R}$ be the function defined as $f(x_1,\ldots,x_{r+1})=\sum_{i=1}^{r+1}x_i^{2}$. Let $C=\{(x_1,\ldots,x_{r+1})\in \mathbb{R}^{r+1}:\sum_{i=1}^{r+1}x_i=M \text{~and~} 0\leq a\leq x_i\leq b\}$ for some fixed $a$, $b$ and $ra+b\leq M\leq (r+1)b$. Then, we have
\begin{equation*}
f(x_1,\ldots,x_2)\leq r_0b^2+(r-r_0)a^2+(M-r_0b-(r-r_0)a)^2,
\end{equation*}
where $r_0$ is the largest integer satisfying $M-r_0b\geq (r+1-r_0)a$. Moreover, the equality holds if and only $x_1=\ldots=x_{r_0}=b$, $x_{r_0+1}=M-r_0b-(r-r_0)a$ and $x_{r_0+2}=\ldots=x_{r+1}=a$, up to isomorphism.
\end{theorem}
\begin{proof}
Noted that $f(x_1,\ldots,x_{r+1})=\sum_{i=1}^{r+1}(x_i-a)^{2}+2Ma-(r+1)a^2$ over $C$, therefore, we only need to prove the case when $a=0$.

When $a=0$, $C$ is actually the polyhedral convex set in $(r+1)$-dimensional cube $[0,b]^{r+1}$ cut by the hyperplane $\sum_{i=1}^{r+1}x_i=M$. Clearly, $f$ is a convex function. Therefore, by Corollary 32.3.4 in \cite{Rockafellar}, the supremum of $f$ relative to $C$ is attained at one of the extreme points of $C$. Denote $r_1=\lfloor\frac{M}{b}\rfloor$, since coordinates of the extreme points of $C$ all have the form: $x_{i_1}=\ldots=x_{i_{r_1}}=b$, $x_{i_{r_1+1}}=M-r_1b$ and the rest $x_i$s all equal to zero. Therefore, we have $f(x_1,\ldots,x_{r+1})\leq r_1b^2+(M-r_1b)^2$ and the equality holds if and only if $(x_1,\ldots,x_{r+1})$ is an aforementioned extreme point of $C$.
\end{proof}

Armed with all these lemmas whose proofs we defer until later in this section, we now show how to deduce Theorem \ref{main0}.

\begin{proof}[Proof of Theorem \ref{main0}]

We prove the theorem by induction on $r$.

Consider the base case: $r=0$. In this case, $\mathcal{F}$ contains at most one full $t$-star. By Lemma \ref{mainl04}, we know that $\mathcal{F}$ has a $t$-cover $U_t$ with size $|U_t|\leq t(4r+5)$. W.l.o.g., assume that $U_t=[m]$. From Lemma \ref{mainl03} and Claim 9 in the proof of Lemma \ref{mainl04}, we know that as one of the most popular $t$-sets appearing in $\mathcal{F}$, $[t]$ has the degree $|\mathcal{F}([t])|\geq (1-\frac{2k^t}{C_t})|\mathcal{F}|$. If $\mathcal{F}\neq \mathcal{F}([t])$, we must have $[t]\subsetneq [m]$. For $i\in[m]\setminus [t]$, by Claim 11 in the proof of Lemma \ref{mainl04}, we have $|\mathcal{F}(\{1,\ldots,t-1,i\})|\geq |\mathcal{F}([t])|-\frac{2k}{C_1}|\mathcal{F}|$. Thus,
\begin{align*}
|\mathcal{F}|&\geq |\mathcal{F}([t])|+|\mathcal{F}\{1,\ldots,t-1,i\})|-|\mathcal{F}(\{1,\ldots,t,i\})|\\
&\geq (2-\frac{4k^t}{C_t}-\frac{2k}{C_1})|\mathcal{F}|-{n-(t+1)\choose k-(t+1)}>|\mathcal{F}|,
\end{align*}
a contradiction. Therefore, $\mathcal{F}= \mathcal{F}([t])\subseteq\mathcal{L}_{n,k,t}^{(1)}$.


Now, let $r_0$ be a non-negative integer. Assume that for every $r\leq r_0$, $\mathcal{F}\subseteq {[n]\choose k}$ with size $\sum_{i=t}^{t+r-1}{{n-i}\choose {k-t}}+\delta{{n-(r+t)}\choose {k-t}}$ satisfying $\mathcal{I}(\mathcal{F})=\mathcal{MI}(\mathcal{F})$ is isomorphic to some $\mathcal{F}_0$ with the structure $\mathcal{L}_{n,k,t}^{(r)}\subseteq \mathcal{F}_0 \subseteq \mathcal{L}_{n,k,t}^{(r+1)}$. We shall prove that this also holds when $r=r_0+1$ by induction on $n$, $k$ and $t$.


Assume that $\mathcal{F}$ can be covered by a single element $x_0\in[n]$, i.e., there exists an $x_0\in [n]$ such that $\mathcal{F}=\mathcal{F}(x_0)$. Then, by identity (\ref{basic_id}), the optimality of $\mathcal{F}$ is guaranteed by the new family
\begin{equation*}
\partial_{k-1}(\mathcal{F}(\overline{x_0}))=\{F\setminus\{x_0\}:F\in\mathcal{F}\}\subseteq {{[n]\setminus \{x_0\}}\choose k-1}
\end{equation*}
with the same size as $\mathcal{F}$. Noted that
\begin{equation*}
\sum_{i=t}^{t+r-1}{{n-i}\choose {k-t}}+\delta{{n-(r+t)}\choose {k-t}}=\sum_{i=t-1}^{(t-1)+r-1}{{(n-1)-i}\choose {(k-1)-(t-1)}}+\delta{{(n-1)-(r+t-1)}\choose {(k-1)-(t-1)}}.
\end{equation*}
Thus, the result follows from the induction hypothesis for the case $n-1$, $k-1$, $r=r_0+1$ and $t-1$. In view of this, to complete the proof, we only need to show that all $F\in \mathcal{F}$ share one common element. If $\mathcal{F}$ contains no full $t$-star, this result follows from Lemma \ref{mainl05}. Therefore, the case left is when $\mathcal{F}$ contains at least one full $t$-star. For the induction process, we can assume that $\mathcal{F}\subseteq{[n]\choose k}$ with size $\sum_{i=t}^{t+r-1}{{n-i}\choose {k-t}}+\delta{{n-(r+t)}\choose {k-t}}>{n-t\choose k-t}$.

Suppose $\mathcal{F}$ contains $p$ full $t$-star $\mathcal{Y}_1,\dots,\mathcal{Y}_p$ for some integer $1\leq p\leq r$. Let $\mathcal{F}_1=\cup_{i=1}^{p}\mathcal{Y}_i$ and $\mathcal{F}_2=\mathcal{F}\setminus \mathcal{F}_1$. Denote $\mathcal{G}_0$ as the optimal subfamily in ${[n]\choose k}$ of size $\sum_{i=t}^{t+r-p-1}{n-i\choose k-t}+\delta'{n-(t+r-p)\choose k-t}$ with respect to $\mathcal{I}(\mathcal{G}_0)$, where $\frac{r\delta}{r+1}\leq \delta'\leq \delta$. Since $p\geq 1$, by induction on $r$, we know that $\mathcal{L}_{n,k,t}^{(r-p)}\subseteq \mathcal{G}_0\subseteq\mathcal{L}_{n,k,t}^{(r-p+1)}$. Thus,
\begin{equation*}
\mathcal{I}(\mathcal{F}_0)\leq\mathcal{I}(\mathcal{G}_0)\leq(t-1)|\mathcal{G}_0|^2+(r-p+\delta^2+\frac{1}{C_t}){n-t\choose k-t}^2,
\end{equation*}
for any $\mathcal{F}_0\subseteq \mathcal{F}$ with size $|\mathcal{F}_0|=|\mathcal{G}_0|$. Therefore, by Lemma \ref{mainl07},
\begin{itemize}
  \item for all $i\neq j\in [p]$, $|Y_i\cap Y_j|=t-1$;
  \item for at least $(1-\frac{2r^2kt}{C_t})|\mathcal{F}_2|$ $k$-sets $F\in\mathcal{F}_2$, $|F\cap Y_i|=t-1$ for each $1\leq i\leq p$;
  \item $\mathcal{I}(\mathcal{F}_2)\geq (t-1)|\mathcal{F}_2|^2+(r-p+\delta^2-\frac{4kr^2p^2}{C_t}){n-t\choose k-t}^2$.
\end{itemize}
Furthermore, by Lemma \ref{mainl03}, the most popular $t$-set $A$ appearing in $\mathcal{F}_2$ has degree $|\mathcal{F}_2(A)|\geq \frac{r-p+\delta^2}{2t(r-p+\delta)^2}|\mathcal{F}_2|$. Denote $\mathcal{F}'_2=\{F\in \mathcal{F}_2: |F\cap Y_i|=t-1 \text{ for each } 1\leq i\leq p\}$ and $\mathcal{F}_3=\mathcal{F}_2\setminus \mathcal{F}'_2$, we have $\mathcal{F}=\mathcal{F}_1\sqcup\mathcal{F}'_2\sqcup\mathcal{F}_3$. In the following, we shall determine all the possible structures of $\mathcal{F}_1$ and $\mathcal{F}_2$ through discussions of the value of $p$ and as a consequence, we will have $\mathcal{F}_3=\emptyset$.


To guarantee that $|Y_i\cap Y_j|=t-1$ for all $i\neq j\in [p]$, there are only two possible cases:
\begin{itemize}
  \item The first case: $p\leq t+1$, up to isomorphism, $Y_i\in {[t+1]\choose t}$ for all $i\in [p]$.
\end{itemize}

When $p\leq2$, structures of $Y_i$s are the same as the second case, which will be discussed later.

When $p\geq 3$, since $|F\cap Y_i|=t-1$ for all $F\in \mathcal{F}'_2$ and $i\in[p]$, we know that for each $F\in \mathcal{F}_2'$, $|F\cap [t+1]|=t$. Therefore, assume that $Y_i=[t+1]\setminus \{i\}$ and $\mathcal{F}'_2=\cup_{i=p+1}^{t+1}\mathcal{H}_i$, where $\mathcal{H}_{i}$ is the $t$-star in $\mathcal{F}_2$ with core $[t+1]\setminus \{i\}$. Since $|\mathcal{F}_3|\leq \frac{2r^2kt}{C_t}|\mathcal{F}_2|$, w.l.o.g., we can assume $A=[t+1]\setminus \{p+1\}$ as the most popular $t$-set in $\mathcal{F}_2$. Then, $|\mathcal{F}_2(A)|=|\mathcal{H}_{p+1}|\geq\frac{r-p+\delta^2}{2t(r-p+\delta)^2}|\mathcal{F}_2|$. Clearly, $A\in {[t+1]\choose t}$. Denote $\mathcal{Z}_0=\{j\in [p+1,t+1]: |\mathcal{H}_j|\geq \frac{|\mathcal{F}_2|}{C_1}\}$, we claim that $\mathcal{F}'_2=\cup_{j\in \mathcal{Z}_0}\mathcal{H}_{j}$. Otherwise, assume that there exists a $G_0\in \mathcal{H}_{j_0}$ for some $j_0\in [p+1,t+1]\setminus\mathcal{Z}_0$. Since $\mathcal{F}_2$ contains no full $t$-star, by replacing $G_0$ with some $F$ containing $A$, we have
\begin{align*}
\mathcal{I}(F,\tilde{\mathcal{F}})-\mathcal{I}(G_0,\mathcal{F})&=\mathcal{I}(F,\mathcal{F}_1)-\mathcal{I}(G_0,\mathcal{F}_1)+\mathcal{I}(F,\tilde{\mathcal{F}'_2})-\mathcal{I}(G_0,\mathcal{F}'_2)+\mathcal{I}(F,\mathcal{F}_3)-\mathcal{I}(G_0,\mathcal{F}_3),
\end{align*}
where $\tilde{\mathcal{F}}=\mathcal{F}_1\sqcup\tilde{\mathcal{F}'_2}\sqcup\mathcal{F}_3$ and $\tilde{\mathcal{F}'_2}=\mathcal{F}'_2\setminus\{G_0\}\sqcup\{F\}$ is the new ``$\mathcal{F}'_2$'' after shifting. Noted that for every $x\in [p]$, $\mathcal{F}_1(x)=|\mathcal{F}_1|-|\mathcal{Y}_x|+{n-(t+1)\choose k-(t+1)}$; for every $x\in [p+1,t+1]$, $|\mathcal{F}_1(x)|=|\mathcal{F}_1|$; and for every $x\in [t+2,n]$, $|\mathcal{F}_1(x)|={n-(t+2)\choose k-(t+2)}+p{n-(t+2)\choose k-(t+1)}$. Therefore, $\mathcal{I}(F,\mathcal{F}_1)-\mathcal{I}(G_0,\mathcal{F}_1)=0$ and
\begin{align*}
\mathcal{I}(F,\tilde{\mathcal{F}})-\mathcal{I}(G_0,\mathcal{F})&=\mathcal{I}(F,\tilde{\mathcal{F}_2})-\mathcal{I}(G_0,\mathcal{F}_2)\geq\sum_{x\in F}|\tilde{\mathcal{F}'_2}(x)|-\sum_{x\in G_0}(|\mathcal{F}'_2(x)|+|\mathcal{F}_3(x)|)\\
&\geq (|\mathcal{F}'_2|-|\mathcal{H}(j_0)|)-(|\mathcal{F}'_2|-|\mathcal{H}(p+1)|)-\sum_{x\in G_0\setminus[t+1]}|\mathcal{F}'_2(x)|-k|\mathcal{F}_3(x)|\\
&\geq \frac{r-p+\delta^2}{2t(r-p+\delta)^2}|\mathcal{F}_2|-\frac{k-t}{C_1}|\mathcal{F}_2|-\frac{2r^2k^2t}{C_t}|\mathcal{F}_2|>0,
\end{align*}
a contradiction, where the third inequality follows from $|\mathcal{H}_{j_0}|\leq \frac{|\mathcal{F}_2|}{C_1}$, $|\mathcal{H}_{p+1}|\geq \frac{r-p+\delta^2}{2t(r-p+\delta)^2}|\mathcal{F}_2|$, $|\mathcal{F}_2'(x)|\leq (t+1){n-(t+1)\choose k-(t+1)}\leq \frac{|\mathcal{F}_2|}{C_1}$ for each $x\in [n]\setminus[t+1]$ and $|\mathcal{F}_3(x)|\leq \frac{2r^2kt}{C_t}|\mathcal{F}_2|$.

Moreover, if $\mathcal{F}_3\neq\emptyset$, let $G_1\in \mathcal{F}_3$. Again, we can replace $G_1$ with some $F$ containing $A$. Denote $\tilde{\mathcal{F}}=\mathcal{F}_1\sqcup\tilde{\mathcal{F}'_2}\sqcup\tilde{\mathcal{F}_3}$ as the new family.
Since $|G_1\cap Y_{i_0}|\leq t-2$ for some $i_0\in[p]$, thus $|G_1\cap[t+1]|\leq t-1$. When $p+1\notin G_1$, we have $(G_1\cap [t+1])\subsetneq A$. Assume $x_0\in A\setminus G_1$, since $\mathcal{F}_1\cup \mathcal{F}'_2=\cup_{i=1}^{t+1}\mathcal{H}_i$ ($\mathcal{H}_i=\mathcal{Y}_i$ for $i\in [p]$), we have
\begin{align*}
\sum_{x\in F}|\tilde{\mathcal{F}}(x)|-\sum_{x\in G_1}|\mathcal{F}(x)|&\geq|\mathcal{F}_2(x_0)|-\sum_{x\in G_1\setminus[t+1]}|\mathcal{F}(x)|\\
&\geq \frac{r-p+\delta^2}{2t(r-p+\delta)^2}|\mathcal{F}_2|-k(t+1){n-(t+1)\choose k-(t+1)}-k|\mathcal{F}_3|\\
&\geq \frac{r-p+\delta^2}{2t(r-p+\delta)^2}|\mathcal{F}_2|-\frac{3r^2k^2t}{C_t}|\mathcal{F}_2|>0.
\end{align*}
When $p+1\in G_1$, we have $|G_1\cap A|\leq t-2$. Assume $x_1,x_2\in A\setminus G_1$, we have
\begin{align*}
\sum_{x\in F}|\tilde{\mathcal{F}}(x)|-\sum_{x\in G_1}|\mathcal{F}(x)|&\geq|\mathcal{F}(x_1)|+|\mathcal{F}(x_2)|-|\mathcal{F}(p+1)|-\sum_{x\in G_1\setminus[t+1]}|\mathcal{F}(x)|\\
&\geq (|\mathcal{F}|-|\mathcal{H}_{x_1}|-|\mathcal{F}_3|)+(|\mathcal{F}|-|\mathcal{H}_{x_2}|-|\mathcal{F}_3|)-(|\mathcal{F}|-|\mathcal{H}_{p+1}|)-\frac{3r^2k^2t}{C_t}|\mathcal{F}_2|\\
&\geq (|\mathcal{F}|-|\mathcal{H}_{x_1}|-|\mathcal{H}_{x_2}|)+|\mathcal{H}_{p+1}|-\frac{5r^2k^2t}{C_t}|\mathcal{F}_2|.
\end{align*}
Since $|\mathcal{F}|-|\mathcal{H}_{x_1}|-|\mathcal{H}_{x_2}|\geq -|\mathcal{H}_{x_1}\cap\mathcal{H}_{x_2}|\geq -{n-(t+1)\choose k-(t+1)}$ and $|\mathcal{H}_{p+1}|=|\mathcal{F}_2(A)|$, thus the above inequality is lower bounded by $\frac{r-p+\delta^2}{2t(r-p+\delta)^2}|\mathcal{F}_2|-\frac{6r^2k^2t}{C_t}|\mathcal{F}_2|>0$. Both cases contradict the optimality of $\mathcal{F}$. Therefore, $\mathcal{F}_3=\emptyset$ and $\mathcal{F}=\cup_{i\in \mathcal{Z}_0\cup [p]}\mathcal{H}_i$, where for $i\in [p]$, $\mathcal{H}_i=\mathcal{Y}_i$ is the full $t$-star with core $[t+1]\setminus\{i\}$.

\begin{itemize}
  \item  The second case: all $Y_i$s share $t-1$ elements in common.
\end{itemize}

The second case is much more complicated. W.l.o.g., assume $Y_i=\{1,2,\ldots,t-1,t-1+i\}$ for $1\leq i\leq p$. To guarantee that $|F\cap Y_i|=t-1$ for all $F\in \mathcal{F}'_2$ and $i\in[p]$, we have the following claim:

\textbf{Claim 14.} Either $p\leq2$ and $\mathcal{F}'_2=\cup_{j=1}^{t+1-p}\mathcal{H}_j$, or $\mathcal{F}'_2=\cup_{i=t+p}^{l}\mathcal{G}_i$ for some $l\in[t+p,n]$, where $\mathcal{H}_j$ is a $t$-star with core $[t+1]\setminus \{j\}$ and $\mathcal{G}_i$ is a $t$-star with core $\{1,\ldots,t-1,i\}$.

\begin{proof}

\begin{itemize}
  \item Case I: $p\geq 3$.
\end{itemize}
Assume that there exists an $F_0\in\mathcal{F}'_2$ such that $|F_0\cap [t-1]|\leq t-2$. Since $|F_0\cap Y_i|=t-1$, we have $[t,t+p-1]\subseteq F_0$ and $|F_0\cap [t-1]|=t-2$. Thus, such $F_0$ contains at least $t+1$ fixed elements. By the choice of $\delta$, this indicates that
\begin{equation*}
|\{F_0\in \mathcal{F}'_2:|F_0\cap [t-1]|\leq t-2\}|\leq p(t-1){n-(t+1)\choose k-(t+1)}<\frac{|\mathcal{F}'_2|}{C_t}.
\end{equation*}
Noted that $|\mathcal{F}_3|\leq \frac{2r^2kt}{C_t}$, therefore, at least $|\mathcal{F}_2|(1-\frac{3r^2kt}{C_t})$ $k$-sets in $\mathcal{F}_2$ contain $[t-1]$.

W.l.o.g., assume $\mathcal{F}_{2}([t-1])=\cup _{j=t+p}^{l}\mathcal{G}_j$.
Since the most popular $t$-set $A$ in $\mathcal{F}_2$ satisfies $|\mathcal{F}_2(A)|\geq\frac{r-p+\delta^2}{2t(r-p+\delta)^2}|\mathcal{F}_2|$, thus $[t-1]\subseteq A$. Assume that $A=\{1,2,\ldots,t-1,t+p\}$ and denote $\mathcal{Z}_1=\{j\in [t+p,l]:|\mathcal{G}_j|\geq\frac{|\mathcal{F}_2|}{C_1}\}$. Since
\begin{align*}
|\mathcal{F}_2([t-1])|\geq\sum_{j\in \mathcal{Z}_1}|\mathcal{G}_j|-\sum_{j_1\neq j_2\in \mathcal{Z}_1}|\mathcal{G}_{j_1}\cap\mathcal{G}_{j_2}|,
\end{align*}
we have $|\mathcal{Z}_1|\leq 2C_1$. By the optimality of $\mathcal{F}$, we claim that $\mathcal{F}_2([t-1])=\cup_{j\in \mathcal{Z}_1}\mathcal{G}_{j}$. Otherwise, assume that there exists a $G_0\in \mathcal{G}_{j_0}$ for some $j_0\in [t+p,l]\setminus \mathcal{Z}_1$. Since $\mathcal{F}_2$ contains no full $t$-star, by replacing $G_0$ with some $F$ containing $A$, we have
\begin{align*}
\mathcal{I}(F,\tilde{\mathcal{F}})-\mathcal{I}(G_0,\mathcal{F})&=\mathcal{I}(F,\mathcal{F}_1)-\mathcal{I}(G_0,\mathcal{F}_1)+\mathcal{I}(F,\tilde{\mathcal{F}'_2})-\mathcal{I}(G_0,\mathcal{F}'_2)+\mathcal{I}(F,\mathcal{F}_3)-\mathcal{I}(G_0,\mathcal{F}_3),
\end{align*}
where $\tilde{\mathcal{F}}=\mathcal{F}_1\sqcup\tilde{\mathcal{F}'_2}\sqcup\mathcal{F}_3$ and $\tilde{\mathcal{F}'_2}=\mathcal{F}'_2\setminus\{G_0\}\sqcup\{F\}$. The structure of $\mathcal{F}_1$ indicates that for every $x\in [t-1]$, $|\mathcal{F}_1(x)|=|\mathcal{F}_1|$; for every $x\in [t,t+p-1]$, $|\mathcal{F}_1(x)|={n-t\choose k-t}$; and for every $x\in [t+p,n]$, $|\mathcal{F}_1(x)|=\sum_{i=t+1}^{t+p}{n-i\choose k-(t+1)}$. Therefore, $\mathcal{I}(F,\mathcal{F}_1)-\mathcal{I}(G_0,\mathcal{F}_1)=0$ and
\begin{align*}
\mathcal{I}(F,\tilde{\mathcal{F}})-\mathcal{I}(G_0,\mathcal{F})&\geq\sum_{x\in F}|\tilde{\mathcal{F}'_2}(x)|-\sum_{x\in G_0}(|\mathcal{F}'_2(x)|+|\mathcal{F}_3(x)|)\\
&\geq \frac{r-p+\delta^2}{2t(r-p+\delta)^2}|\mathcal{F}_2|-\sum_{x\in G_0\setminus[t-1]}|\mathcal{F}'_2(x)|-k|\mathcal{F}_3|\\
&\geq \frac{r-p+\delta^2}{2t(r-p+\delta)^2}|\mathcal{F}_2|-\frac{k-t+1}{C_1}|\mathcal{F}'_2|-\frac{2r^2k^2t}{C_t}|\mathcal{F}_2|>0,
\end{align*}
a contradiction.

Recall that $|F_0\cap[t-1]|=t-2$, again, we can replace $F_0$ with some $F$ containing $A$ and denote the new family as $\tilde{\mathcal{F}}=\mathcal{F}_1\sqcup\tilde{\mathcal{F}'_2}\sqcup\mathcal{F}_3$. The above argument actually proved that $\mathcal{F}_2([t-1])$ has a small $t$-cover, since $|\mathcal{G}_i\cap\mathcal{G}_j|\leq {n-(t+1)\choose k-(t+1)}$ for $i\neq j\in \mathcal{Z}_1$, this enables us to control the value of $\mathcal{I}(F_0,\mathcal{F}_2)$. Thus, we have
\begin{align*}
\mathcal{I}(F,\tilde{\mathcal{F}})-\mathcal{I}(F_0,\mathcal{F})&=\mathcal{I}(F,\mathcal{F}_1)-\mathcal{I}(F_0,\mathcal{F}_1)+\mathcal{I}(F,\tilde{\mathcal{F}'_2})-\mathcal{I}(F_0,\mathcal{F}'_2)+\mathcal{I}(F,\mathcal{F}_3)-\mathcal{I}(F_0,\mathcal{F}_3)\\
&\geq|\mathcal{F}_1|-p{n-t\choose k-t}+(1+\frac{r-p+\delta^2}{2t(r-p+\delta)^2}-\frac{3r^2kt+p}{C_t})|\mathcal{F}_2|-\sum_{x\in F_0\setminus[t+p-1]}|\mathcal{F}'_2(x)|-k|\mathcal{F}_3|\\
&\geq (1+\frac{r-p+\delta^2}{2t(r-p+\delta)^2}-\frac{3r^2kt+p+p^2}{C_t})|\mathcal{F}_2|-\sum_{x\in F_0\cap\mathcal{Z}_1}|\mathcal{F}_2'(x)|-\frac{k}{C_1}|\mathcal{F}_2|-\frac{2r^2k^2t}{C_t}|\mathcal{F}_2|\\
&\geq (\frac{r-p+\delta^2}{2t(r-p+\delta)^2}-\frac{7r^2k^2t}{C_t}-\frac{2k}{C_1})|\mathcal{F}_2|>0,
\end{align*}
where the third inequality follows from $|\mathcal{F}_2'|\geq \sum_{x\in \mathcal{Z}_1}|\mathcal{F}_2'(x)|-\sum_{x\neq y\in \mathcal{Z}_1}|\mathcal{F}_2'(x)\cap \mathcal{F}_2'(y)|$ and $|\mathcal{F}_2'(x)\cap \mathcal{F}_2'(y)|\leq |\mathcal{G}_x\cap\mathcal{G}_y|+\frac{3r^2kt}{C_t}|\mathcal{F}_2|$. This contradicts the optimality of $\mathcal{F}$ and thus disproves the existence of $F_0$. Therefore, $\mathcal{F}'_2=\mathcal{F}_2([t-1])=\cup _{j=t+p}^{l}\mathcal{G}_j$. 

\begin{itemize}
  \item Case II: $p=1$.
\end{itemize}

Assume $Y_1=[t]$. By Lemma \ref{mainl04}, $\mathcal{F}$ has a $t$-cover $U_t$ of size $|U_t|\leq t(4r+5)$. According to the proof of Lemma \ref{mainl04}, $U_{t}=\bigcup_{A\in\mathcal{X}_{t}}A$, where $\mathcal{X}_{t}=\{A\in{[n]\choose t}:|\mathcal{F}(A)|\geq \frac{|\mathcal{F}|}{C_t}\}$, we have $[t]\subseteq U_t$. Therefore, denote $\mathcal{A}=\{A\in {U_t\choose t}, |A\cap[t]|=t-1\}$, we have $|\mathcal{A}|\leq t^2(4r+5)$ and $\mathcal{F}'_2=\cup_{A\in \mathcal{A}}\mathcal{F}(A)$. First, for each $A\in \mathcal{A}$,
\begin{align*}
\mathcal{I}(\mathcal{F}(A))&=\sum_{x\in[n]}|\mathcal{F}(A\cup\{x\})|^2\leq t|\mathcal{F}(A)|^2+(n-t){n-(t+1)\choose k-(t+1)}^2;
\end{align*}
and for $A_1\neq A_2\in \mathcal{A}$,
\begin{align*}
\mathcal{I}(\mathcal{F}(A_1),\mathcal{F}(A_2))&=\sum_{F_1\in \mathcal{F}(A_1)}\sum_{F_{2}\in \mathcal{F}(A_2)}|F_1\cap F_2|=\sum_{F_1\in \mathcal{F}(A_1)}\sum_{x\in F_1}|\mathcal{F}(A_2\cup\{x\})|\\
&\leq |A_1\cap A_2||\mathcal{F}(A_1)||\mathcal{F}(A_2)|+t(|\mathcal{F}(A_1)|+|\mathcal{F}(A_2)|){n-(t+1)\choose k-(t+1)}+2k{n-(t+1)\choose k-(t+1)}^2.
\end{align*}
Therefore, we have
\begin{align}\label{ineq07.1}
\mathcal{I}(\mathcal{F}'_2)&\leq\sum_{A\in \mathcal{A}}\mathcal{I}(\mathcal{F}(A))+2\sum_{A_1\neq A_2\in\mathcal{A}}\mathcal{I}(\mathcal{F}(A_1),\mathcal{F}(A_2))\nonumber\\
&\leq t\sum_{A\in \mathcal{A}}|\mathcal{F}(A)|^2+2(\sum_{A_1\neq A_2\in\mathcal{A}}|A_1\cap A_2||\mathcal{F}(A_1)||\mathcal{F}(A_2)|)+\frac{1}{C_t}{n-t\choose k-t}^2.
\end{align}
Moreover, since $\sum_{A\in \mathcal{A}}|\mathcal{F}(A)|<|\mathcal{F}_2|+\frac{{n-t\choose k-t}}{3C_t}<(r-p+\delta+\frac{1}{3C_t}){n-t\choose k-t}$ and $|\mathcal{F}(A)|\leq {n-t\choose k-t}$ for each $A\in \mathcal{A}$, by Theorem \ref{convex_property}, we have
\begin{align}\label{ineq07.2}
\sum_{A\in \mathcal{A}}|\mathcal{F}(A)|^2<(r-p+\delta^2+\frac{1}{C_t}){n-t\choose k-t}^2.
\end{align}
On the other hand, $|\mathcal{F}_3|\leq \frac{2r^2kt}{C_t}|\mathcal{F}_2|$ leads to
\begin{align}\label{ineq07.3}
\mathcal{I}(\mathcal{F}'_2)&\geq \mathcal{I}(\mathcal{F}_2)-\frac{5r^4k^2t}{C_t}{n-t\choose k-t}^2\geq (t-1)|\mathcal{F}_2|^2+(r-p+\delta^2-\frac{9r^4 k^2t}{C_t}){n-t\choose k-t}^2.
\end{align}
Thus, combining (\ref{ineq07.1}), (\ref{ineq07.2}) and (\ref{ineq07.3}) together, we have
\begin{align*}
(t-1)|\mathcal{F}_2|^2-\frac{10r^4k^2t}{C_t}{n-t\choose k-t}^2\leq(t-1)(\sum_{A\in\mathcal{A}}|\mathcal{F}(A)|)^2+2\sum_{A_1\neq A_2\in\mathcal{A}}(|A_1\cap A_2|-(t-1))|\mathcal{F}(A_1)||\mathcal{F}(A_2)|.
\end{align*}


Noted that $(t-2)\leq |A_1\cap A_2|\leq (t-1)$ and $|\mathcal{F}(A_1)\cap\mathcal{F}(A_2)|\leq {n-(t+1)\choose k-(t+1)}$ for $A_1\neq A_2\in\mathcal{A}$, the above inequality actually shows
\begin{align}\label{ineq07.4}
\sum_{A_1\in\mathcal{A}}\sum_{\substack{A_2\in \mathcal{A},\\|A_1\cap A_2|=t-2}}|\mathcal{F}(A_1)||\mathcal{F}(A_2)|\leq\frac{5r^4k^2t}{C_t}{n-t\choose k-t}^2.
\end{align}
Denote $\mathcal{A}_1=\{A\in\mathcal{A}:|\mathcal{F}(A)|\geq \frac{{n-t\choose k-t}}{C_1}\}$, since $|\mathcal{A}|\leq t^2(4r+5)\leq 9rt^2$, we have $|\bigcup_{A\in \mathcal{A}_1}\mathcal{F}(A)|\geq |\mathcal{F}_2'|-\frac{t^2(4r+5)}{C_1}{n-t\choose k-t}$. For each $A\in \mathcal{A}_1$, (\ref{ineq07.4}) shows that \begin{align*}
|\bigcup_{\substack{B\in \mathcal{A},\\|B\cap A|<t-1}}\mathcal{F}(A)|\leq \frac{5r^4k^2tC_1}{C_t}{n-t\choose k-t}.
\end{align*}
Therefore, we can remove at most $\frac{45r^5k^2t^3C_1}{C_t}{n-t\choose k-t}$ $k$-sets from $\bigcup_{A\in \mathcal{A}_1}\mathcal{F}(A)$ and obtain a subfamily $\mathcal{A}'\subseteq \mathcal{A}_1$ such that $|A_1\cap A_2|=t-1$ for all $A_1\neq A_2\in \mathcal{A}'$. By the choice of $C_1$ and $C_t$, $|\cup_{A\in \mathcal{A}'}\mathcal{F}(A)|\geq (1-\frac{10r^5 k^2t^3}{C_1(r-1+\delta)})|\mathcal{F}_2|$. Therefore, 
similar to the structures of $Y_i$s, either $\mathcal{A}'\subseteq {[t+1]\choose t}$ or all $A\in \mathcal{A}'$ and $[t]$ share $t-1$ common elements. Thus, either $\cup_{A\in \mathcal{A}'}\mathcal{F}(A)=\cup_{j=1}^{t}\mathcal{H}_j$ or $\cup_{A\in \mathcal{A}'}\mathcal{F}(A)=\cup_{i=t+1}^{l}\mathcal{G}_i$. For $B\in \mathcal{A}\setminus\mathcal{A}'$, if $|B\cap A|=t-1$ for all $A\in \mathcal{A}'$, then either $B\in {[t+1]\choose t}$ or $[t-1]\subseteq B$. Therefore, $\cup_{A\in \mathcal{A}'\cup\{B\}}\mathcal{F}(A)$ has the same structure as $\cup_{A\in \mathcal{A}'}\mathcal{F}(A)$. W.l.o.g., we can assume that for each $B\in \mathcal{A}\setminus\mathcal{A}'$ there exists some $A\in {\mathcal{A}'}$ such that $|A\cap B|<t-1$.

When $\cup_{A\in \mathcal{A}'}\mathcal{F}(A)=\cup_{j=1}^{t}\mathcal{H}_j$, if $\mathcal{A'}\neq \mathcal{A}$, let $G_0\in \mathcal{F}(A_0)$ for some $A_0\in \mathcal{A}\setminus\mathcal{A}'$, we have $|A_0\cap [t]|=|A_0\cap [t+1]|=t-1$. W.l.o.g., assume the most popular $t$-set in $\mathcal{F}_2$ is $[t+1]\setminus\{t\}$ and $A_0\cap [t+1]=[t]\setminus\{i_0\}$. Since $\mathcal{F}_2$ contains no full $t$-star, we can replace $G_0$ with some $F$ containing $[t+1]\setminus\{t\}$. Denote the new family as $\tilde{\mathcal{F}}$, we have
\begin{align*}
\sum_{x\in F}|\tilde{\mathcal{F}}(x)|-\sum_{x\in G_0}|\mathcal{F}(x)|&\geq|\mathcal{F}(t+1)|+|\mathcal{F}(i_0)|-|\mathcal{F}(t)|-\sum_{x\in G_0\setminus[t+1]}|\mathcal{F}(x)|\\
&\geq|\mathcal{F}_2(t+1)|-\sum_{x\in G_0\setminus[t+1]}|\mathcal{F}_2(x)|\\
&\geq \frac{r-1+\delta^2}{2t(r-1+\delta)^2}|\mathcal{F}_2|-\frac{k(t+1)}{C_t}|\mathcal{F}_2|-\frac{10r^5k^3t^3}{C_1(r-1+\delta)}|\mathcal{F}_2|-k|\mathcal{F}_3|\\
&\geq \frac{r-1+\delta^2}{2t(r-1+\delta)^2}|\mathcal{F}_2|-\frac{15r^5k^3t^3}{C_1(r-1+\delta)}|\mathcal{F}_2|>0,
\end{align*}
a contradiction.

When $\cup_{A\in \mathcal{A}'}\mathcal{F}(A)=\cup_{i=t+1}^{l}\mathcal{G}_i$, if $\mathcal{A'}\neq \mathcal{A}$, let $G_0\in \mathcal{F}(A_0)$ for some $A_0\in \mathcal{A}\setminus\mathcal{A}'$, we have $|A_0\cap [t]|=t-1$ and $A_0\cap [t]\neq [t-1]$. With a similar shifting argument as above, we can also reach a contradiction.

Therefore, we have $\mathcal{A}=\mathcal{A}'$ and this indicates that either $\mathcal{F}'_2=\cup_{j=1}^{t}\mathcal{H}_j$, or $\mathcal{F}'_2=\cup_{i=t+1}^{l}\mathcal{G}_i$.

\begin{itemize}
  \item Case III: $p=2$.
\end{itemize}

Assume that $Y_1=[t+1]\setminus\{t+1\}$ and $Y_2=[t+1]\setminus\{t\}$, since $|F\cap Y_i|=t-1$ for all $F\in\mathcal{F}'_2$ and $i\in[2]$, $\mathcal{F}'_2$ must have the following hybrid structure:
\begin{equation*}
\mathcal{F}'_2=\mathcal{F}_{21}\sqcup\mathcal{F}_{22},
\end{equation*}
where $\mathcal{F}_{21}=\cup_{i=t+2}^{l}\mathcal{G}_i$ denotes the part of $k$-sets containing $[t-1]$,
$\mathcal{F}_{22}=\cup_{j=1}^{t-1}\mathcal{H}_j$ denotes the part that contains a $t$-set from $[t+1]$. To guarantee the optimality of $\mathcal{F}$, we claim that either $\mathcal{F}_{21}=\emptyset$, or $\mathcal{F}_{22}=\emptyset$. Our discussion is divided into the following three parts.

\begin{itemize}
  \item When $|\mathcal{F}_{22}|\leq \frac{|\mathcal{F}_2|}{C_1}$, we have $|\mathcal{F}_{21}|\geq (1-\frac{1}{C_1})|\mathcal{F}_2|$. Similar to the case when $p\geq 3$, using shifting arguments, one can prove that $\mathcal{F}_2([t-1])=\mathcal{F}_{21}$ has a small $t$-cover and then derive $\mathcal{F}'_{2}=\mathcal{F}_{21}$ by contradiction.
  \item When $|\mathcal{F}_{22}|\geq (1-\frac{1}{C_1})|\mathcal{F}_2|$, since the most popular $t$-set $A$ in $\mathcal{F}_2$ satisfies $|\mathcal{F}_2(A)|\geq \frac{r-p+\delta^2}{3t(r-p+\delta)^2}|\mathcal{F}_2|>\frac{|\mathcal{F}_2|}{C_1}$, w.l.o.g., assume that $A=[t+1]\setminus\{1\}$. If there exists a $G_0\in \mathcal{F}_{21}$, since $\mathcal{F}_2$ contains no full $t$-star, we can replace $G_0$ with some $F$ containing $A$. Denote the resulting new family as $\tilde{\mathcal{F}}=\mathcal{F}_1\sqcup\tilde{\mathcal{F}'_2}\sqcup\mathcal{F}_3$, then
      \begin{align*}
      \mathcal{I}(F,\tilde{\mathcal{F}})-\mathcal{I}(G_0,\mathcal{F})&=\mathcal{I}(F,\mathcal{F}_1)-\mathcal{I}(G_0,\mathcal{F}_1)+\mathcal{I}(F,\tilde{\mathcal{F}_2}')-\mathcal{I}(G_0,\mathcal{F}_2)+\mathcal{I}(F,\mathcal{F}_3)-\mathcal{I}(G_0,\mathcal{F}_3)\\
      &\geq 2{n-t\choose k-t}-|\mathcal{F}_1|+\sum_{i=2}^{t-1}(|\mathcal{F}_2|-|\mathcal{H}_j|)+2|\mathcal{F}_{22}|-\sum_{i=1}^{t-1}(|\mathcal{F}_2|-|\mathcal{H}_j|)-\frac{3r^2kt}{C_t}|\mathcal{F}_2|\\
      &\geq |\mathcal{H}_1|+(1-\frac{2+kt}{C_1})|\mathcal{F}_{2}|>0,
      \end{align*}
      a contradiction. Therefore, $\mathcal{F}'_{2}=\mathcal{F}_{22}$.
  \item When $\frac{|\mathcal{F}_2|}{C_1}<|\mathcal{F}_{22}|< (1-\frac{1}{C_1})|\mathcal{F}_2|$, assume that $|\mathcal{F}_{22}|=(l_2+\delta_2){n-t\choose k-t}$, where $l_2$ is a non-negative integer and $\delta_2\in [0,1)$. By the structure of $\mathcal{F}_{22}$, we have
      \begin{align*}
      \mathcal{I}(\mathcal{F}_{22})&\leq\sum_{i=1}^{t-1}(|\mathcal{F}_{22}|-|\mathcal{H}_i|)^2+2|\mathcal{F}_{22}|^2+(n-t-1)(t-1)^2{n-(t+1)\choose k-(t+1)}^2\\
      &\leq (t-1)|\mathcal{F}_{22}|^2+\sum_{i=1}^{t-1}|\mathcal{H}_{i}|^{2}+\frac{t^2}{2C_t}{n-t\choose k-t}^2\\
      &\leq (t-1)|\mathcal{F}_{22}|^2+(l_2+\delta_2^2+\frac{t^2}{C_t}){n-t\choose k-t}^2,
      \end{align*}
      where the last inequality follows from Theorem \ref{convex_property}. Combining with the lower bound of $\mathcal{I}(\mathcal{F}_2')$ from (\ref{ineq07.3}), we have
      \begin{align*}
      &\mathcal{I}(\mathcal{F}_{21})+2\mathcal{I}(\mathcal{F}_{21},\mathcal{F}_{22})=\mathcal{I}(\mathcal{F}'_2)-\mathcal{I}(\mathcal{F}_{22})\\
      \geq &(t-1)|\mathcal{F}_{21}|^2+2(t-1)|\mathcal{F}_{21}||\mathcal{F}_{22}|+(r-2+\delta^2-l_2-\delta_2^2-\frac{10r^4k^2t^2}{C_t}){n-t\choose k-t}^2.
      \end{align*}
      On the other hand, we have
      \begin{align*}
      \mathcal{I}(\mathcal{F}_{21},\mathcal{F}_{22})=\sum_{F_1\in \mathcal{F}_{21},F_2\in\mathcal{F}_{22}}|F_1\cap F_2|&\leq (t-2)|\mathcal{F}_{21}||\mathcal{F}_{22}|+\sum_{F_1\in \mathcal{F}_{21},F_2\in\mathcal{F}_{22}}|(F_1\cap F_2)\setminus[t+1]|\\
      &\leq (t-2)|\mathcal{F}_{21}||\mathcal{F}_{22}|+(k-t+1)(t-1){n-(t+1)\choose k-(t+1)}|\mathcal{F}_{21}|\\
      &\leq (t-2+\frac{kt}{C_t(l_2+\delta_2)})|\mathcal{F}_{21}||\mathcal{F}_{22}|.
      \end{align*}
      Thus,
      \begin{equation}\label{ineq22}
      \mathcal{I}(\mathcal{F}_{21})\geq (t-1)|\mathcal{F}_{21}|^2+(2-\frac{2kt}{C_t(l_2+\delta_2)})|\mathcal{F}_{21}||\mathcal{F}_{22}|+(r-2+\delta^2-l_2-\delta_2^2-\frac{10r^4k^2t^2}{C_t}){n-t\choose k-t}^2.
      \end{equation}
      Based on this lower bound, by Lemma \ref{mainl03}, the most popular $t$-set $A'$ in $\mathcal{F}_{21}$ has degree $|\mathcal{F}_{21}(A')|\geq\frac{|\mathcal{F}_{21}|}{3t(r+1)}\geq \frac{|\mathcal{F}_2|}{3t(r+1)C_1}$. W.l.o.g., assume that $A'=\{1,\ldots,t-1,t+2\}$ and denote $\mathcal{Z}_{2}=\{i\in [t+2,l]:|\mathcal{G}_i|\geq\frac{|\mathcal{F}_2|}{C_t}\}$. Then, $|\mathcal{Z}_{2}|\leq 2C_t$. Similar to the case $p\geq 3$, using shifting arguments, we can prove that $\mathcal{F}_{21}=\cup_{i\in \mathcal{Z}_{2}}\mathcal{G}_{i}$.
      Based on this structure of $\mathcal{F}_{21}$, we have
      \begin{align*}
      \mathcal{I}(\mathcal{F}_{21})&\leq (t-1)|\mathcal{F}_{21}|^2+\sum_{i\in \mathcal{Z}_{2}}|\mathcal{G}_i|^2+n|\mathcal{Z}_{2}|{n-(t+1)\choose k-(t+1)}^2\\
      &\leq (t-1)|\mathcal{F}_{21}|^2+[(r-2+\delta^2-l_2-\delta_2^2)+(2\delta_2^2-2\delta_2\delta+\min\{0, 2(\delta-\delta_2)\})+\frac{2}{C_t}]{n-t\choose k-t}^2,
      \end{align*}
      where the second inequality follows from Theorem \ref{convex_property} and the choice of $n$. Since $2\delta_2^2-2\delta_2\delta+\min\{0, 2(\delta-\delta_2)\}\leq 0$, by the choice of $C_1$ and $C_t$, the above upper bound is always strictly less than the lower bound given by (\ref{ineq22}), a contradiction. Therefore, when $\frac{|\mathcal{F}_2|}{C_1}<|\mathcal{F}_{22}|< (1-\frac{1}{C_1})|\mathcal{F}_2|$, $\mathcal{I}(\mathcal{F})$ can not be optimal.
\end{itemize}

Therefore, for all three cases, either $\mathcal{F}'_2=\cup_{j=1}^{t-1}\mathcal{H}_j$ or $\mathcal{F}'_2=\cup_{i=t+2}^{l}\mathcal{G}_i$. This completes the proof of the claim.
\end{proof}

With the same proof as that for the first case, when $\mathcal{F}'_2=\cup_{i=t+p}^{l}\mathcal{G}_i$ in the second case, we can also prove that $\mathcal{F}_2'$ is consisted of large $t$-stars. Denote $\mathcal{Z}=\{i\in [t+p,n]: |\mathcal{G}_i|\geq \frac{|\mathcal{F}_2|}{C_1}\}$, we claim that $\mathcal{F}'_2=\cup_{i\in \mathcal{Z}}\mathcal{G}_i$. Otherwise, assume that there exists a $G_0\in\mathcal{G}_{i_0}$ for some $i_0\notin \mathcal{Z}$. Since the most popular $t$-set A in $\mathcal{F}_2$ has degree $|\mathcal{F}_2(A)|\geq \frac{r-p+\delta^2}{2t(r-p+\delta)^2}|\mathcal{F}_2|$ and $\mathcal{F}_2$ contains no full $t$-star, we can replace $G_0$ with some $F\notin \mathcal{F}$ containing $A$. With a similar counting argument as the proof of Claim 14 for $p\geq 3$, this process strictly increases $\mathcal{I}(\mathcal{F})$, a contradiction. Thus, $\mathcal{F}'_2=\cup_{i\in \mathcal{Z}}\mathcal{G}_i$. Moreover, noted that $|\mathcal{F}_2|\geq \sum_{i\in \mathcal{Z}}|\mathcal{G}_i|-\sum_{i\neq j\in \mathcal{Z}}|\mathcal{G}_i\cap \mathcal{G}_j|$, we have $|\mathcal{Z}|\leq 2C_1$ and $l\leq 2C_1+t+p$.

Now, we show that $\mathcal{F}_3=\emptyset$ for the second case.

When $p\leq 2$ and $\mathcal{F}'_2=\cup_{j=1}^{t+1-p}\mathcal{H}_j$, we have $\mathcal{F}_1\cup\mathcal{F}'_2=\cup_{j=1}^{t+1}\mathcal{H}_j$ which is same as the structure of $\mathcal{F}_1\cup\mathcal{F}'_2$ in the first case when $p\geq 3$. Since the proof of $\mathcal{F}_3=\emptyset$ in the first case only depends on the structure of $\mathcal{F}_1\cup\mathcal{F}'_2$ and is unrelated with the value of $p$, therefore, with the same argument we have $\mathcal{F}_3=\emptyset$.

When $\mathcal{F}'_2=\cup_{i=t+p}^{l}\mathcal{G}_i$, we have $\mathcal{F}_1\cup\mathcal{F}'_2=\cup_{i=t}^{l}\mathcal{G}_i$, where $\mathcal{G}_i=\mathcal{Y}_{i-t+1}$ for $t\leq i\leq t+p-1$. If $\mathcal{F}_3\neq\emptyset$, let $G_1\in \mathcal{F}_3$. Replace $G_1$ with some $F$ containing $A$ and denote $\tilde{\mathcal{F}}=\mathcal{F}_1\sqcup\tilde{\mathcal{F}'_2}\sqcup\tilde{\mathcal{F}_3}$ as the new family. Noted $|G_1\cap Y_{i_0}|\leq t-2$ for some $i_0\in[p]$, thus either $|G_1\cap[t-1]|\leq t-3$ and $t-1+i_0\in G_1$, or $|G_1\cap[t-1]|\leq t-2$ and $t-1+i_0\notin G_1$. When $t-1+i_0\in G_1$, let $x_1,x_2\in [t-1]\setminus G_1$. Since $|\mathcal{F}|-|\mathcal{F}_3|\geq \sum_{i=t}^{l}|\mathcal{G}_i|-l^2{n-(t+1)\choose k-(t+1)}$,we have
\begin{align*}
\sum_{x\in F}|\tilde{\mathcal{F}}(x)|-\sum_{x\in G_1}|\mathcal{F}(x)|&\geq|\mathcal{F}(x_1)|+|\mathcal{F}(x_2)|-\sum_{x\in G_1\setminus[t-1]}|\mathcal{F}(x)|\\
&\geq 2|\mathcal{F}|-2|\mathcal{F}_3|-\sum_{i\in G_1\cap[t,l]}|\mathcal{G}_i|-kl{n-(t+1)\choose k-(t+1)}-k|\mathcal{F}_3|\\
&\geq |\mathcal{F}|-(l^2+kl){n-(t+1)\choose k-(t+1)}-(k+2)|\mathcal{F}_3|>0.
\end{align*}
When $t-1+i_0\notin G_1$, let $x_1\in [t-1]\setminus G_1$, since $|\mathcal{G}_{t-1+i_0}|=|\mathcal{Y}_{i_0}|={n-t\choose k-t}$, we have
\begin{align*}
\sum_{x\in F}|\tilde{\mathcal{F}}(x)|-\sum_{x\in G_1}|\mathcal{F}(x)|&\geq|\mathcal{F}(x_1)|-\sum_{x\in G_1\setminus[t-1]}|\mathcal{F}(x)|\\
&\geq |\mathcal{F}|-|\mathcal{F}_3|-\sum_{i\in G_1\cap[t,l]}|\mathcal{G}_i|-kl{n-(t+1)\choose k-(t+1)}-k|\mathcal{F}_3|\\
&\geq |\mathcal{G}_{t-1+i_0}|-(l^2+kl){n-(t+1)\choose k-(t+1)}-(k+1)|\mathcal{F}_3|>0.
\end{align*}
Both cases contradict the optimality of $\mathcal{F}$. Therefore, $\mathcal{F}_3=\emptyset$ and $\mathcal{F}=\cup_{i=t}^{l}\mathcal{G}_i$.

Finally, we derive the basic outlines of $\mathcal{F}$: $\mathcal{F}=\cup_{i=t}^{l}\mathcal{G}_i$ or $\mathcal{F}=\cup_{j=1}^{t+1}\mathcal{H}_j$.

When $\mathcal{F}=\cup_{i=t}^{l}\mathcal{G}_i$, all $F\in \mathcal{F}$ share $t-1$ common elements. When $\mathcal{F}=\cup_{j=1}^{t+1}\mathcal{H}_j$, if there exists some $j_0\in [t+1]$ such that $|\mathcal{H}_{j_0}|=0$, then all $F\in \mathcal{F}$ contain $j_0$. If $\mathcal{H}_j\neq \emptyset$ for all $j\in[t+1]$, from the proof of $\mathcal{F}_3=\emptyset$ in the first case, $|\mathcal{H}_j|\geq \frac{|\mathcal{F}_2|}{C_1}>\frac{\delta}{2C_1}{n-t\choose k-t}$. By Lemma \ref{mainl08}, there exists a family $\mathcal{F}_0$ of size $|\mathcal{F}|$ such that $\mathcal{L}_{n,k,t}^{(r)}\subseteq \mathcal{F}_0\subseteq \mathcal{L}_{n,k,t}^{(r+1)}$ and $\mathcal{I}(\mathcal{F}_0)>\mathcal{I}(\mathcal{F})$, a contradiction. Therefore, all $F\in \mathcal{F}$ always share one common element and the result follows from the induction.

This completes the proof.
\end{proof}

It remains to prove the lemmas. First, with the same strategy as that of Lemma \ref{mainl01}, we give a proof of Lemma \ref{mainl03}.
\begin{proof}[Proof of Lemma \ref{mainl03}]
Fix $t\geq 1$, let $C_t=2^{2^{t-1}-1}\cdot 10^{2^{t+2}-2}\cdot (k^2t^4(r+1)^7)^{2^{t-1}}$ and take $\mathcal{X}_{t}=\{A\in{[n]\choose t}:|\mathcal{F}(A)|\geq \frac{|\mathcal{F}|}{C_t}\}$ as the family of moderately popular $t$-sets appearing in $\mathcal{F}$. First, we show that $\mathcal{X}_{t}$ can not be very large.

\textbf{Claim 7.} $|\mathcal{X}_t|< 2C_t$.
\begin{proof}
Suppose not, let $\mathcal{X}_0$ be a subfamily of $\mathcal{X}_t$ with size $2C_t$, then we have
\begin{align}\label{ineq08}
|\mathcal{F}|\geq |\bigcup_{A\in \mathcal{X}_0}\mathcal{F}(A)|&\geq \sum_{A\in \mathcal{X}_0}|\mathcal{F}(A)|-\sum_{A\neq B\in \mathcal{X}_0}|\mathcal{F}(A\cup B)|\\
&\geq 2|\mathcal{F}|-{|\mathcal{X}_0|\choose 2}{{n-(t+1)}\choose {k-(t+1)}}.\nonumber
\end{align}
Since $|\mathcal{F}|=|\mathcal{L}_{n,k,t}^{(r)}|+\delta{n-(r+t)\choose k-t}$, we know that
\begin{align}\label{ineq09}
|\mathcal{F}|&\geq {r\choose 1}{{n-t}\choose {k-t}}-{r\choose 2}{n-(t+1)\choose k-(t+1)}+\delta{n-(r+t)\choose k-t}\\
&\geq (\frac{nr}{3k}+\delta\frac{n-(r+k)}{k-t}(1-\frac{k(r+k)}{n-t})){n-(t+1)\choose k-(t+1)}.\nonumber
\end{align}

Combining (\ref{ineq08}) and (\ref{ineq09}) together, we have $|\mathcal{F}|\geq (2-\frac{6{C_t}^{2}k}{n(r+\delta)})|\mathcal{F}|$, which contradicts the requirement of $n$. Thus, the claim holds.
\end{proof}

Now, we complete the proof by proving the following claim.

\textbf{Claim 8.} There is an $A_0\in \mathcal{X}_t$ such that $|\mathcal{F}(A_0)|\geq \frac{r+\delta^2}{2t(r+\delta)^{2}}|\mathcal{F}|$.

\begin{proof}
Since
\begin{align}\label{ineq10}
\mathcal{I}(\mathcal{F}) &\geq (t-1)|\mathcal{F}|^{2}+(r+\delta^2-\epsilon_0){n-t\choose k-t}^2 \\
&\geq(t-1+\frac{9(r+\delta^2)}{10(r+\delta)^2})|\mathcal{F}|^2, \nonumber
\end{align}
where the second inequality follows from that $\mathcal{L}_{n,k}=\cup_{i=t}^{t+r}\mathcal{G}_i$, where $\mathcal{G}_i$ is the full $t$-star with core $[t-1]\cup\{i\}$ for $t\leq i\leq t+r-1$ and $\mathcal{G}_{t+r}$ is contained in the full $t$-star with core $[t-1]\cup\{t+r\}$.

W.l.o.g, assume that $[t]\in \mathcal{X}_t$ is the most popular $t$-subset appearing in $\mathcal{F}$. Noticed that
\begin{equation}\label{ineq11}
\sum_{A\in{[n]\choose t}}|\mathcal{F}(A)|^2=\sum_{F_1\in \mathcal{F}}\sum_{F_2\in \mathcal{F}}{{|F_1\cap F_2|}\choose t}
\end{equation}
and function ${x\choose t}=\frac{x(x-1)\cdots(x-t+1)}{t!}$ is convex when $x\geq t-1$. According to (\ref{ineq10}), we know that $\frac{\mathcal{I}(\mathcal{F})}{|\mathcal{F}|^2}> t-1$. Therefore, by Jensen's inequality, we have
\begin{align}\label{ineq12}
{{\frac{\mathcal{I}(\mathcal{F})}{|\mathcal{F}|^2}}\choose t}\cdot|\mathcal{F}|^2&={{\frac{\sum_{F_1,F_2\in\mathcal{F}}|F_1\cap F_2|}{|\mathcal{F}|^2}}\choose t}\cdot|\mathcal{F}|^2 \nonumber \\
&\leq \sum_{F_1,F_2\in \mathcal{F}}{{|F_1\cap F_2|}\choose t}=\sum_{A\in{[n]\choose t}}|\mathcal{F}(A)|^2.
\end{align}
Since ${x\choose t}$ is increasing in $x$ when $x\geq t-1$, we also have
\begin{equation}\label{ineq12.5}
{{\frac{\mathcal{I}(\mathcal{F})}{|\mathcal{F}|^2}}\choose t}\cdot|\mathcal{F}|^2\geq {{t-1+\frac{9(r+\delta^2)}{10(r+\delta)^2}}\choose t}\cdot|\mathcal{F}|^2\geq \frac{9(r+\delta^2)}{10t(r+\delta)^2}\cdot|\mathcal{F}|^2.
\end{equation}
Therefore, by combining the above inequalities together, we can obtain
\begin{align}\label{ineq13}
\frac{9(r+\delta^2)}{10t(r+\delta)^2}\cdot|\mathcal{F}|^2 &\leq\sum_{A\in{[n]\choose t}}|\mathcal{F}(A)|^2=\sum_{A\in \mathcal{X}_t}|\mathcal{F}(A)|^2+\sum_{A\in {[n]\choose t}\setminus \mathcal{X}_t}|\mathcal{F}(A)|^2 \nonumber\\
&\leq |\mathcal{F}([t])|\cdot\sum_{A\in \mathcal{X}_t}|\mathcal{F}(A)|+\frac{|\mathcal{F}|}{C_t}\cdot\sum_{A\in {[n]\choose t}\setminus \mathcal{X}_t}|\mathcal{F}(A)| \nonumber\\
&\leq |\mathcal{F}([t])|\cdot(|\mathcal{F}|+{|\mathcal{X}_t|\choose 2}{n-(t+1)\choose k-(t+1)})+\frac{|\mathcal{F}|}{C_t}\cdot {k\choose t}\cdot|\mathcal{F}| \nonumber\\
&\leq |\mathcal{F}([t])|\cdot|\mathcal{F}|\cdot(1+\frac{6{C_t}^{2}k}{n(r+\delta)})+\frac{{k\choose t}}{C_t}|\mathcal{F}|^2.
\end{align}
This leads to $|\mathcal{F}([t])|\geq  \frac{r+\delta^2}{2t(r+\delta)^2}|\mathcal{F}|$. Therefore, the claim holds.
\end{proof}
Moreover, when $r=0$, we have $\mathcal{I}(\mathcal{F})\geq t|\mathcal{F}|^2$, which changes the RHS of (\ref{ineq12.5}) to $|\mathcal{F}|^2$. This leads to $|\mathcal{F}([t])|\geq  (1-\frac{2k^t}{C_t})|\mathcal{F}|$.
\end{proof}




Based on Lemma \ref{mainl03}, we turn to the proof of Lemma \ref{mainl04}. Different from the proof of Lemma \ref{mainl02}, according to the definition of $\mathcal{I}(\mathcal{F})$, it seems that the optimality of $\mathcal{F}$ can only guarantee the control of $|\mathcal{F}(x)|$. This is far from enough, since what we want is the control of $|\mathcal{F}(A)|$ for every $A\notin\mathcal{X}_t$. Therefore, besides the moderately popular $t$-sets $A\in\mathcal{X}_t$, we also consider the $t$-sets consisting of elements from every moderately popular $s$-sets ($1\leq s\leq t-1$).


\begin{proof}[Proof of Lemma \ref{mainl04}]


For each $1\leq s\leq t-1$, we define
\begin{equation*}
\mathcal{X}_s=\{A\in {[n]\choose s}:|\mathcal{F}(A)|\geq\frac{|\mathcal{F}|}{C_s}\text{ and } A\nsubseteq B, \text{ for any } B\in \bigcup_{i=s+1}^{t}\mathcal{X}_i\}
\end{equation*}
as the family of moderately popular $s$-sets appearing in $\mathcal{F}$ except those already contained in some moderately popular $(s+1)$-sets, where $C_s=2^{2^{s-1}-1}\cdot 10^{2^{s+2}-2}\cdot (k^2t^4(r+1)^7)^{2^{s-1}}$. Since $\frac{2{C_s}^2}{C_{s+1}}<1$, we claim that $|\mathcal{X}_s|\leq 2C_s$. Otherwise, let $\mathcal{X}_0$ be a subfamily of $\mathcal{X}_s$ with size $2C_s$, we have
\begin{align*}
|\mathcal{F}|\geq |\bigcup_{A\in \mathcal{X}_0}\mathcal{F}(A)|&\geq \sum_{A\in \mathcal{X}_0}|\mathcal{F}(A)|-\sum_{A\neq B\in \mathcal{X}_0}|\mathcal{F}(A\cup B)|.
\end{align*}
A little different from (\ref{ineq08}), since $A,B\in \mathcal{X}_s$ are not contained in any member of $\mathcal{X}_{s+1}$, for $A\neq B\in \mathcal{X}_s$, we have $|\mathcal{F}(A\cup B)|\leq \frac{|\mathcal{F}|}{C_{s+1}}$. Then, through a similar argument as that of Claim 7, we can reach a contradiction.

Let $U_{i}=\bigcup_{A\in\mathcal{X}_{i}}A$ and $U=\bigcup_{1\leq s\leq t}U_s$, for the convenience of our following proof, w.l.o.g., we assume that $U=[m]$ and $|\mathcal{F}(1)|\geq|\mathcal{F}(2)|\geq\ldots\geq |\mathcal{F}(m)|$. Based on this ordering, we have the following claims.

\textbf{Claim 9.} $[t]$ is one of the most popular $t$-sets appearing in $\mathcal{F}$. Moreover, if $\mathcal{F}$ contains a full $t$-star, then the core of this $t$-star is $[t]$ and $[t+1]\setminus \{t\}$ is one of the most popular $t$-sets appearing in $\mathcal{F}\setminus{\mathcal{F}([t])}$.

\begin{proof}
Let $A_0\neq [t]$ be one of the most popular $t$-sets appearing in $\mathcal{F}$, by Lemma \ref{mainl03}, $A_0\subseteq [m]$. Assume that $1\notin A_0$, we consider the new family $\mathcal{S}_{a_0,1}(\mathcal{F}(A_0))$, where $a_0\in A_0\setminus[t]$. If there exists some $F\in \mathcal{S}_{a_0,1}(\mathcal{F}(A_0))\setminus\mathcal{F}$, we can replace its preimage $\mathcal{S}_{1,a_0}(F)$ in $\mathcal{F}$ with $F$. Denote the new family as $\mathcal{F}'$, then $\mathcal{I}(\mathcal{F}')-\mathcal{I}(\mathcal{F})$ is
\begin{align*}
\sum_{x\in F}|\mathcal{F}'(x)|-\sum_{(F\setminus \{1\})\cup\{a_0\}}|\mathcal{F}(x)|&= |\mathcal{F}(1)|+1-|\mathcal{F}(a_0)|>0.
\end{align*}
Therefore, by the optimality of $\mathcal{F}$, $\mathcal{S}_{a_0,1}(\mathcal{F}(A_0))\subseteq \mathcal{F}(1)$.

Let $A_1=A_0\setminus\{a_0\}\cup\{1\}$, we know that $|\mathcal{F}(A_1)|=|\mathcal{S}_{a_0,1}(\mathcal{F}(A_0))|=|\mathcal{F}(A_0)|$. Now, assume that $2\notin A_1$, let $A_2=A_1\setminus\{a_1\}\cup\{2\}$ for some $a_1\in A_2\setminus[t]$. With a similar argument, we have $|\mathcal{F}(A_2)|=|\mathcal{F}(A_0)|$. By repeating this process, finally, we can obtain $|\mathcal{F}([t])|=|\mathcal{F}(A_0)|$.

If $\mathcal{F}$ contains one full $t$-star $\mathcal{Y}_1$, we have $r\geq 1$. From the analysis above, we have $\mathcal{Y}_1=\mathcal{F}([t])$. Denote $\mathcal{F}_1=\mathcal{Y}_1$ and $\mathcal{F}_2=\mathcal{F}\setminus\mathcal{F}_1$. Since $\mathcal{I}({\mathcal{F}})=\mathcal{MI}(\mathcal{F})\geq (t-1)|\mathcal{F}|^{2}+(r+\delta^2){n-t\choose k-t}^2$ and $\mathcal{I}(\mathcal{F}_1)\leq(t+\frac{1}{C_t}){n-t\choose k-t}^2$, we have
\begin{align*}
\mathcal{I}(\mathcal{F}_2)+2\mathcal{I}(\mathcal{F}_1,\mathcal{F}_2)\geq (t-1)|\mathcal{F}_2|^2+2(t-1)|\mathcal{F}_1||\mathcal{F}_2|+(r-1+\delta^2-\frac{1}{C_t}){n-t\choose k-t}^2.
\end{align*}
Moreover, noted that for each $F_2\in \mathcal{F}_2$, $|F_2\cap[t]|\leq t-1$, we have
\begin{align*}
\mathcal{I}(\mathcal{F}_1,\mathcal{F}_2)&=\sum_{F_1\in \mathcal{F}_1}\sum_{F_2\in \mathcal{F}_2}|F_1\cap F_2|\leq \sum_{F_1\in \mathcal{F}_1}\sum_{F_2\in \mathcal{F}_2}(|[t]\cap F_2|+|F_1\cap (F_2\setminus [t])|)\\
&\leq (t-1)|\mathcal{F}_1||\mathcal{F}_2|+k|\mathcal{F}_2|{n-(t+1)\choose k-(t+1)}\leq (t-1)|\mathcal{F}_1||\mathcal{F}_2|+\frac{kr}{C_t}{n-t\choose k-t}^2.
\end{align*}
Combining the above two inequalities, we have
\begin{equation}\label{ineq14.5}
\mathcal{I}(\mathcal{F}_2)\geq (t-1)|\mathcal{F}_2|^2+(r-1+\delta^2-\frac{3kr}{C_t}){n-t\choose k-t}^2.
\end{equation}
By Lemma \ref{mainl03}, we know that the most popular $t$-set $A$ satisfies $|\mathcal{F}_2(A)|\geq \frac{r-1+\delta^2}{2t(r-1+\delta)^2}|\mathcal{F}_2|\geq\frac{\delta}{3t(r+\delta)}|\mathcal{F}| $.

Let $B_0\neq [t+1]\setminus \{t\}$ be one of the most popular $t$-sets appearing in $\mathcal{F}\setminus{\mathcal{F}([t])}$. Then, $|\mathcal{F}(B_0)|\geq |\mathcal{F}_2(A)|$, which indicates that $B_0\subseteq [m]$. Similarly, assume that $1\notin B_0$ and consider the new family $\mathcal{S}_{b_0,1}(\mathcal{F}(B_0))$, where $b_0\in B_0\setminus[t]$. Using a same shifting argument, we have $\mathcal{S}_{b_0,1}(\mathcal{F}(B_0))\subseteq \mathcal{F}(1)$. Then, repeating this process for $2,3,\ldots,t-1$ and $t+1$ successively, we can obtain $|\mathcal{F}([t+1]\setminus \{t\})|=|\mathcal{F}(B_0)|$.
\end{proof}

For $A,B\in {[m]\choose t}$, let $A\setminus B=\{a_1,a_2,\ldots,a_l\}$ and $B\setminus A=\{b_1,b_2,\ldots,b_l\}$, where $a_1\leq \ldots\leq a_l$ and $b_1\leq\ldots\leq b_l$ for some $0\leq l\leq t$. From the given ordering that $|\mathcal{F}(1)|\geq\ldots\geq |\mathcal{F}(m)|$, the shifting argument in Claim 9 actually shows that if $a_i\geq b_i$ for all $1\leq i\leq l$, then $|\mathcal{F}(A)|\geq |\mathcal{F}(B)|$.

\textbf{Claim 10.} $[m]$ is a $t$-cover of $\mathcal{F}$.
\begin{proof}
Suppose not, there exists an $F_0\in \mathcal{F}$ such that $|F_0\cap [m]|\leq t-1$. Actually, by the definition of $[m]$, we know that for each $x\in F_0\setminus[m]$ and every $1\leq s\leq t$, there is no moderately popular $s$-set containing $x$. Thus, for each $x\in F_0\setminus[m]$, we have $|\mathcal{F}(x)|\leq \frac{|\mathcal{F}|}{C_1}$.

Since $\mathcal{F}$ contains at most one full $t$-star, by Claim 9, we can assume $[t]$ as the core of this only full $t$-star in $\mathcal{F}$ (if exists). Thus, we can replace $F_0$ with some $F\in{[n]\choose k}$ containing $[t+1]\setminus \{t\}$. Denote the new family as $\mathcal{F}'$. Noted that $|\mathcal{F}(t+1)|\geq |\mathcal{F}([t+1]\setminus\{t\})|\geq\frac{\delta}{3t(r+\delta)}|\mathcal{F}|$. Thus, we have
\begin{align*}
\sum_{x\in F}|\mathcal{F}'(x)|-\sum_{x\in F_0}|\mathcal{F}(x)|&\geq \sum_{i=1}^{t-1}|\mathcal{F}(i)|+|\mathcal{F}(t+1)|-\sum_{x\in F_0\cap [m]}|\mathcal{F}(x)|-\frac{k-t+1}{C_1}|\mathcal{F}|\\
&\geq \frac{\delta}{3t(r+\delta)}|\mathcal{F}|-\frac{k-t+1}{C_1}|\mathcal{F}|>0,
\end{align*}
which contradicts the optimality of $\mathcal{F}$. Thus, $[m]$ is a $t$-cover of $\mathcal{F}$. 
\end{proof}

Now, we only need to show that for each $i\in [m]$, $i$ is contained in some $A\in \mathcal{X}_t$. For $i\in [t]$, this easily follows from the fact that $[t]\in\mathcal{X}_t$. For $t+1\leq i\leq m$, we have the following claim.

\textbf{Claim 11.} When $\mathcal{F}$ contains no full $t$-star, for $t+1\leq i\leq m$, $|\mathcal{F}(i)|\geq |\mathcal{F}(t)|-\frac{k}{C_1}|\mathcal{F}|$ and $|\mathcal{F}(\{1,2,\ldots,t-1,i\})|\geq |\mathcal{F}([t])|-\frac{2k}{C_1}|\mathcal{F}|$. When $\mathcal{F}$ contains one full $t$-star, for $t+1\leq i\leq m$, $|\mathcal{F}(i)|\geq |\mathcal{F}(t+1)|-\frac{k}{C_1}|\mathcal{F}|$ and $|\mathcal{F}(\{1,2,\ldots,t-1,i\})|\geq |\mathcal{F}([t+1]\setminus\{t\})|-\frac{2k}{C_1}|\mathcal{F}|$.
\begin{proof}
First, similar to Claim 9, by a shifting argument, we can prove that $\{1,2,\ldots,t-1,i\}$ has the largest degree in $\mathcal{F}$ among all $t$-sets in ${[m]\choose t}$ containing $i$. This indicates that
\begin{align*}
|\mathcal{F}(\{1,\ldots,t-1,i\})|\geq \frac{|\mathcal{F}(i)|}{{m\choose t-1}}.
\end{align*}
By the definition of $U$, we know that $|\mathcal{F}(i)|\geq \frac{|\mathcal{F}|}{C_t}$ and $m=|U|\leq \sum_{i=1}^{t}2iC_i$. Therefore,
\begin{align*}
|\mathcal{F}(\{1,\ldots,t-1,i\})|\geq \frac{|\mathcal{F}|}{(3tC_t)^{t}}.
\end{align*}
If for every $F\in\mathcal{F}(\{1,\ldots,t-1,i\})$, $|F\cap [m]|\geq t+1$, then the size of $\mathcal{F}(\{1,\ldots,t-1,i\})$ shall be upper bounded by
\begin{align*}
|\mathcal{F}(\{1,\ldots,t-1,i\})|\leq m{{n-(t+1)}\choose k-(t+1)},
\end{align*}
which contradicts the above lower bound since $n$ is very large. Thus, there is an $F_0\in \mathcal{F}(\{1,\ldots,t-1,i\})$ such that $|F_0\cap [m]|=t$.

When $\mathcal{F}$ contains no full $t$-star, we can replace this $F_0$ with an $F\notin \mathcal{F}$ containing $[t]$. Denote the new family as $\mathcal{F}'$, then, we have
\begin{align*}
\sum_{x\in F}|\mathcal{F}'(x)|-\sum_{x\in F_0}|\mathcal{F}(x)|&\geq \sum_{i=1}^{t}|\mathcal{F}(i)|-\sum_{x\in F_0\cap [m]}|\mathcal{F}(x)|-\frac{k-t+1}{C_1}|\mathcal{F}|\\
&=|\mathcal{F}(t)|-|\mathcal{F}(i)|-\frac{k-t+1}{C_1}|\mathcal{F}|.
\end{align*}
Therefore, $|\mathcal{F}(i)|\geq|\mathcal{F}(t)|-\frac{k-t+1}{C_1}|\mathcal{F}|$ follows from the optimality of $\mathcal{F}$.

When $\mathcal{F}$ contains one full $t$-star $\mathcal{Y}_1$, according to Claim 9, the core of $\mathcal{Y}_1$ is $[t]$ and $[t+1]\setminus\{t\}$ is the most popular $t$-set with degree less than ${n-t\choose k-t}$. 
Thus, we can replace $F_0$ with some $F'\notin \mathcal{F}$ containing $[t+1]\setminus \{t\}$ and a same argument leads to $|\mathcal{F}(i)|\geq|\mathcal{F}(t+1)|-\frac{k-t+1}{C_1}|\mathcal{F}|$.

Fix $j\in [m]\setminus\{i\}$. For any $B\in {[m]\setminus \{i\}\choose t}$ containing $j$,  we know that $\mathcal{F}(B)\subseteq \mathcal{F}(j)$ and $\mathcal{F}(B\setminus \{j\}\cup\{i\})\subseteq \mathcal{F}(i)$. Moreover, from our previous analysis, $|\mathcal{F}(B)|\geq|\mathcal{F}(B\setminus \{j\}\cup\{i\})|$ for all $j+1\leq i\leq m$. Since for each $j\in[m]$,
\begin{align}\label{ineq21.5}
\sum_{B\in {[m]\choose t}, j\in B}|\mathcal{F}(B)|-\sum_{B_1\neq B_2\in {[m]\choose t}, j\in B_1,B_2}|\mathcal{F}(B_1\cup B_2)|\leq|\mathcal{F}(j)|\leq \sum_{B\in {[m]\choose t},j\in B}|\mathcal{F}(B)|.
\end{align}
When $\mathcal{F}$ contains no full $t$-star, take $j=t$. For $t+1\leq i\leq m$, combining with $|\mathcal{F}(i)|\geq|\mathcal{F}(t)|-\frac{k-t+1}{C_1}|\mathcal{F}|$, we have
\begin{align*}
\sum_{B\in {[m]\choose t},i\in B}|\mathcal{F}(B)|\geq \sum_{B\in {[m]\choose t}, t\in B}|\mathcal{F}(B)|-\sum_{B_1\neq B_2\in {[m]\choose t}, t\in B_1,B_2}|\mathcal{F}(B_1\cup B_2)|-\frac{k-t+1}{C_1}|\mathcal{F}|.
\end{align*}
Noted for each $B\in{[m]\choose t}$ containing $\{t\}$, $|\mathcal{F}(B)|\geq |\mathcal{F}(B\setminus\{t\}\cup\{i\})|$. Thus, we have
\begin{align*}
\sum_{B_1\neq B_2\in {[m]\choose t}, t\in B_1,B_2}|\mathcal{F}(B_1\cup B_2)|+\frac{k-t+1}{C_1}|\mathcal{F}|&\geq \sum_{B\in {[m]\choose t}, t\in B}|\mathcal{F}(B)|-\sum_{B'\in {[m]\choose t},i\in B'}|\mathcal{F}(B')|\\
&\geq|\mathcal{F}([t])|-|\mathcal{F}(\{1,2,\ldots,t-1,i\})|,
\end{align*}
where the last inequality follows from the one to one correspondence between $B\in{[m]\choose t}$ containing $\{t\}$ and $B'\in{[m]\choose t}$ containing $\{i\}$. Since $|\mathcal{F}(B_1\cup B_2)|\leq {{n-(t+1)}\choose k-(t+1)}$, by the choice of $n$, we have
\begin{align*}
|\mathcal{F}(\{1,2,\ldots,t-1,i\})|&\geq|\mathcal{F}([t])|-\frac{k-t+1}{C_1}|\mathcal{F}|-{m\choose t-1}^{2}{{n-(t+1)}\choose k-(t+1)}\\
&\geq |\mathcal{F}([t])|-\frac{2k}{C_1}|\mathcal{F}|.
\end{align*}
When $\mathcal{F}$ contains one full $t$-star $\mathcal{Y}_1=\mathcal{F}([t])$, through a similar estimation, we have $|\mathcal{F}(\{1,2,\ldots,t-1,i\})|\geq |\mathcal{F}([t+1]\setminus\{t\})|-\frac{2k}{C_1}|\mathcal{F}|$.
\end{proof}

When $\mathcal{F}$ contains no full $t$-star, noted that $|\mathcal{F}([t])|\geq \frac{r+\delta^2}{2t(r+\delta)^2}|\mathcal{F}|$, by Claim 11, $|\mathcal{F}(\{1,2,\ldots,t-1,i\})|\geq \frac{r+\delta^2}{3t(r+\delta)^2}|\mathcal{F}|$ for $t+1\leq i\leq m$. Therefore, $\{1,2,\ldots,t-1,i\}\in \mathcal{X}_t$ and since \begin{align*}
|\mathcal{F}|&\geq\sum_{i=t}^{m}|\mathcal{F}(\{1,2,\ldots,t-1,i\})|-\sum_{i\neq j\in [t,m]}|\mathcal{F}(\{1,2,\ldots,t-1,i,j\})|\\
&\geq\sum_{i=t}^{m}|\mathcal{F}(\{1,2,\ldots,t-1,i\})|-\frac{1}{C_t}|\mathcal{F}|,
\end{align*}
we have $m\leq \frac{4t(r+\delta)^2}{r+\delta^2}+t\leq t(4r+5)$.

When $\mathcal{F}$ contains one full $t$-star which is assumed as $\mathcal{F}([t])$. From Claim 9, $|\mathcal{F}([t+1]\setminus\{t\})|\geq \frac{r-1+\delta^2}{2t(r-1+\delta)^2}|\mathcal{F}_2|$, where $\mathcal{F}_2=\mathcal{F}\setminus \mathcal{F}([t])$. By Claim 11, $|\mathcal{F}(\{1,2,\ldots,t-1,i\})|\geq \frac{r-1+\delta^2}{3t(r-1+\delta)^2}|\mathcal{F}_2|$ for $t+1\leq i\leq m$. Thus, we also have $\{1,2,\ldots,t-1,i\}\in \mathcal{X}_t$. Moreover, since
\begin{align*}
|\mathcal{F}_2|&\geq\sum_{i=t+1}^{m}|\mathcal{F}(\{1,2,\ldots,t-1,i\})|-\sum_{i\neq j\in [t+1,m]}|\mathcal{F}(\{1,2,\ldots,t-1,i,j\})|\\
&\geq\sum_{i=t+1}^{m}|\mathcal{F}(\{1,2,\ldots,t-1,i\})|-\frac{1}{C_t}|\mathcal{F}_2|,
\end{align*}
we have $m\leq \frac{4t(r-1+\delta)^2}{r-1+\delta^2}+t\leq t(4r+5)$.

This completes the proof of Lemma \ref{mainl04}.
\end{proof}

We now proceed the proof of Lemma \ref{mainl05}. This shows that when $\mathcal{F}$ contains no full $t$-star, all $F\in \mathcal{F}$ share a common element.

\begin{proof}[Proof of Lemma \ref{mainl05}]

For the convenience of the proof, we inherit the assumptions that $[m]=U=U_t$ and $|\mathcal{F}(1)|\geq|\mathcal{F}(2)|\geq\ldots\geq |\mathcal{F}(m)|$ in the proof of Lemma \ref{mainl04}.

Noted that $[m]$ is a $t$-cover of $\mathcal{F}$, first, we have the following claim which says that $|\mathcal{F}(1)|$ is already fairly large.

\textbf{Claim 12.} $|\mathcal{F}(1)|\geq (\frac{t-1}{t}+\frac{r+\delta^2}{t(r+\delta)^2}-\frac{1}{C_t})|\mathcal{F}|$. 

\begin{proof}
By the definition of $\mathcal{I}(\mathcal{F})$ and inequality (\ref{ineq21.5}), for each $A=\{a_1,a_2,\ldots,a_t\}\in {[m]\choose t}$, we have
\begin{align}\label{ineq14}
\mathcal{I}(\mathcal{F}(A),\mathcal{F})&=\sum_{F\in \mathcal{F}(A)}\mathcal{I}(F,\mathcal{F})\leq\sum_{B\in {[m]\choose t}}\sum_{F\in \mathcal{F}(A)}\mathcal{I}(F,\mathcal{F}(B))\nonumber\\
&=\sum_{B\in {[m]\choose t}}\sum_{F\in \mathcal{F}(A)}(\sum_{x\in A\cap B}|\mathcal{F}(B)|+\sum_{x\in F\cap(B\setminus A)}|\mathcal{F}(B)|+\sum_{x\in F\setminus B}|\mathcal{F}(B\cup \{x\})|)\nonumber\\
&\leq \sum_{B\in {[m]\choose t}}\big[|A\cap B||\mathcal{F}(A)||\mathcal{F}(B)|+\sum_{x\in B\setminus A}\sum_{F\in \mathcal{F}(A),x\in F}|\mathcal{F}(B)|+k|\mathcal{F}(A)|{n-(t+1)\choose k-(t+1)}\big]\nonumber\\
&\leq \sum_{B\in {[m]\choose t}}\big[|A\cap B||\mathcal{F}(A)||\mathcal{F}(B)|+(t|\mathcal{F}(B)|+k|\mathcal{F}(A)|){n-(t+1)\choose k-(t+1)}\big]\nonumber\\
&\leq |\mathcal{F}(A)|\cdot\big[\sum_{i=1}^{t}\sum_{B\in {[m]\choose t}, a_i\in B}|\mathcal{F}(B)|+k{m\choose t}{n-(t+1)\choose k-(t+1)}\big]+t{n-(t+1)\choose k-(t+1)}\sum_{B\in {[m]\choose t}}|\mathcal{F}(B)|\nonumber\\
&\leq |\mathcal{F}(A)|\cdot\big[\sum_{i=1}^{t}|\mathcal{F}(a_i)|\cdot(1+\frac{1}{4C_t})+k{m\choose t}{n-(t+1)\choose k-(t+1)}\big]+\frac{2t(k-t)}{n-t}|\mathcal{F}|^2\nonumber\\
&\leq |\mathcal{F}(A)|\cdot\sum_{i=1}^{t}|\mathcal{F}(a_i)|\cdot(1+\frac{1}{4C_t})+\frac{2k^2m^t}{n-t}|\mathcal{F}|^2,
\end{align}
where the last inequality follows from $|\mathcal{F}(A)|\leq |\mathcal{F}(a_i)|\leq |\mathcal{F}|$ and $t\leq m\leq t(4r+5)$. This leads to
\begin{align*}
\mathcal{I}(\mathcal{F})&\leq\sum_{A\in {[m]\choose t}}\mathcal{I}(\mathcal{F}(A),\mathcal{F})\leq \sum_{A\in {[m]\choose t}}|\mathcal{F}(A)|\cdot\sum_{i\in [t]}|\mathcal{F}(i)|\cdot(1+\frac{1}{4C_t})+\frac{2k^2m^{2t}}{n-t}|\mathcal{F}|^2 \nonumber\\
&\leq \sum_{i\in [t]}|\mathcal{F}(i)|\cdot |\mathcal{F}|\cdot (1+\frac{1}{2C_t})+\frac{t|\mathcal{F}|^2}{2C_t}.
\end{align*}
Since $\mathcal{I}(\mathcal{F})\geq \mathcal{I}(\mathcal{L}_{n,k,t}(|\mathcal{F}|))\geq (t-1+\frac{r+\delta^2}{(r+\delta)^2})|\mathcal{F}|^{2}$, thus we have
\begin{align}\label{ineq16}
\sum_{i\in [t]}|\mathcal{F}(i)|&\geq\frac{1}{|\mathcal{F}|}\cdot(\mathcal{I}(\mathcal{F})-\frac{t|\mathcal{F}|^2}{C_t})\geq (t-1+\frac{r+\delta^2}{(r+\delta)^2}-\frac{t}{C_t})|\mathcal{F}|.
\end{align}
This indicates that $|\mathcal{F}(1)|\geq (\frac{t-1}{t}+\frac{r+\delta^2}{t(r+\delta)^2}-\frac{1}{C_t})|\mathcal{F}|$.
%
\end{proof}

Now, according to the size of $|\mathcal{F}(1)|$, we divide our arguments into the following two cases.

\begin{itemize}
  \item When $|\mathcal{F}(1)|\geq (1-\frac{1}{C_t})|\mathcal{F}|$.
\end{itemize}

Assume that there exists an $F_0\in\mathcal{F}$ such that $1\notin F_0$. Since $\mathcal{F}$ contains no full $t$-star, we can replace $F_0$ with some $F$ containing $[t]$. Denote the new family as $\mathcal{F}'$, the gain of this shifting procedure is
\begin{align*}
\sum_{x\in F}|\mathcal{F}'(x)|-\sum_{x\in F_0}|\mathcal{F}(x)|\geq |\mathcal{F}(1)|-\sum_{x\in F_0\setminus [t]}|\mathcal{F}(x)|.
\end{align*}
Since $[m]$ is a $t$-cover of $\mathcal{F}$ and $|\{F\in\mathcal{F}: |F\cap [m]|\geq t+1\}|\leq {m\choose t+1}{n-(t+1)\choose k-(t+1)}$, we have
\begin{align}\label{ineq17.5}
t|\mathcal{F}|\leq\sum_{i\in [m]}|\mathcal{F}(i)|\leq (t+\frac{t}{C_t})|\mathcal{F}|.
\end{align}
Combined with (\ref{ineq16}), this indicates that
\begin{align*}
\sum_{i=t+1}^{m}|\mathcal{F}(i)|\leq (1+\frac{2t}{C_t}-\frac{r+\delta^2}{(r+\delta)^2})|\mathcal{F}|.
\end{align*}
Therefore, we have
\begin{align*}
\sum_{x\in F}|\mathcal{F}'(x)|-\sum_{x\in F_0}|\mathcal{F}(x)|&\geq |\mathcal{F}(1)|-\sum_{x=t+1}^{m}|\mathcal{F}(x)|-\frac{k}{C_1}|\mathcal{F}|\\
&\geq (\frac{r+\delta^2}{(r+\delta)^2}-\frac{3t}{C_t}-\frac{k}{C_1})|\mathcal{F}|>0,
\end{align*}
a contradiction. Therefore, if $|\mathcal{F}(1)|\geq (1-\frac{1}{C_t})|\mathcal{F}|$, then the optimality of $\mathcal{F}$ guarantees that $|\mathcal{F}(1)|=|\mathcal{F}|$, i.e., for all $F\in\mathcal{F}$, $1\in F$.
\begin{itemize}
  \item When $|\mathcal{F}(1)|<(1-\frac{1}{C_t})|\mathcal{F}|$.
\end{itemize}

By the shifting argument in Claim 9, we know that $[t+1]\setminus\{1\}$ has the largest degree in $\mathcal{F}$ among all $t$-sets not containing $1$. Thus, we have $|\mathcal{F}([t+1]\setminus\{1\})|\geq \frac{|\mathcal{F}|}{C_t\cdot {m-1\choose t}}$. Similar to the proof of Claim 11, we can find an $F_0\in \mathcal{F}([t+1]\setminus\{1\})$ such that $|F_0\cap[m]|=t$. Again, replace $F_0$ with some $F$ containing $[t]$ and denote the new family as $\mathcal{F}'$. The gain of this procedure is
\begin{align*}
\sum_{x\in F}|\mathcal{F}'(x)|-\sum_{x\in F_0}|\mathcal{F}(x)|\geq |\mathcal{F}(1)|-|\mathcal{F}(t+1)|-\frac{k}{C_1}|\mathcal{F}|.
\end{align*}
Thus, by the optimality of $\mathcal{F}$, we have $|\mathcal{F}(t+1)|\geq |\mathcal{F}(1)|-\frac{k}{C_1}|\mathcal{F}|$ and this leads to $|\mathcal{F}(2)|\geq\ldots\geq |\mathcal{F}(t)|\geq |\mathcal{F}(1)|-\frac{k}{C_1}|\mathcal{F}|$. By Claim 11 and Claim 12, for all $i\in [m]$, $|\mathcal{F}(i)|\geq |\mathcal{F}(1)|-\frac{2k}{C_1}|\mathcal{F}|\geq (\frac{t-1}{t}+\frac{r+\delta^2}{t(r+\delta)^2}-\frac{3k}{C_1})|\mathcal{F}|$. This leads to
\begin{align*}
\sum_{i\in [m]}|\mathcal{F}(i)|&\geq m(\frac{t-1}{t}+\frac{r+\delta^2}{t(r+\delta)^2}-\frac{3k}{C_1})|\mathcal{F}|\\
&\geq (m-\frac{m}{t}+\frac{m}{t(r+1)}-\frac{3km}{C_1})|\mathcal{F}|.
\end{align*}
When $m\geq t+2$, this gives a lower bound no less than $(t+1-\frac{2}{t}+\frac{t+2}{t(r+1)}-\frac{3k(t+2)}{C_1})>(t+\frac{t}{C_t})$, which contradicts the upper bound in (\ref{ineq17.5}). Therefore, $m\leq t+1$. When $m=t$, $[t]$ is a $t$-cover of $\mathcal{F}$. This means $[t]\in F$ for all $F\in \mathcal{F}$.

When $m=t+1$, the above lower bound is $(t+\frac{t-r}{t(r+1)}-\frac{3k(t+1)}{C_1})$. If $r<t$, it's strictly larger than $t+\frac{t}{C_t}$. Therefore, we have $r\geq t$. Since $[m]=[t+1]$ is a $t$-cover of $\mathcal{F}$, we can assume that $\mathcal{F}=\cup_{i=1}^{t+1}\mathcal{H}_i$, where $\mathcal{H}_{i}$ is the $t$-star in $\mathcal{F}$ with core $[t+1]\setminus \{i\}$. Therefore, $|\mathcal{F}|\leq (t+1){n-(t+1)\choose k-t}+{n-(t+1)\choose k-(t+1)}$. By the choice of $\delta$, this indicates $r<t+1$.
%
Since $|\mathcal{F}(t+1)|\geq |\mathcal{F}(1)|-\frac{k}{C_1}|\mathcal{F}|$, by (\ref{ineq17.5}), we have $|\mathcal{F}(i)|\leq |\mathcal{F}|\cdot(\frac{t}{t+1}+\frac{k+1}{C_1})$ for every $i\in [t+1]$. Denote $\tilde{\mathcal{H}}_j=\mathcal{H}_j\setminus{\mathcal{F}([t+1])}$, we have $|\tilde{\mathcal{H}}_j|=|\mathcal{F}|-|\mathcal{F}(j)|\geq \frac{|\mathcal{F}|}{2(t+1)}$. By Lemma \ref{mainl08}, there exists a family $\mathcal{F}_0$ of size $|\mathcal{F}|$ such that $\mathcal{I}(\mathcal{F})<\mathcal{I}(\mathcal{F}_0)$. This contradicts the optimality of $\mathcal{F}$, therefore, $m\neq t+1$.
Finally, for both cases, we have $\mathcal{F}(1)=\mathcal{F}$, this completes the proof of Lemma \ref{mainl05}.
\end{proof}

Now, we turn to the proof of Lemma \ref{mainl07}, which determines the cross-intersecting structures of all full $t$-stars in $\mathcal{F}$ and their relationships with the remaining $k$-sets of the family.

\begin{proof}[Proof of Lemma \ref{mainl07}]
Let $\mathcal{F}_1=\bigcup_{i=1}^{p}\mathcal{Y}_i$ and $\mathcal{F}_2=\mathcal{F}\setminus \mathcal{F}_1$, then we have
\begin{align*}
\mathcal{I}(\mathcal{F})=\mathcal{I}(\mathcal{F}_1)+2\mathcal{I}(\mathcal{F}_1,\mathcal{F}_2)+\mathcal{I}(\mathcal{F}_2).
\end{align*}
By Lemma \ref{mainl06}, we know that $\sum_{i=t}^{t+p-1}{n-i\choose k-t}\leq |\mathcal{F}_1|< p{n-t\choose k-t}$. Thus, by the choice of $\delta_0$,
\begin{align*}
\sum_{i=t}^{t+r-p-1}{n-i\choose k-t}+\frac{r\delta_0}{r+1}{n-(t+r-p)\choose k-t}<|\mathcal{F}_2|<\sum_{i=t}^{t+r-p-1}{n-i\choose k-t}+\delta_0{n-(t+r-p)\choose k-t}.
\end{align*}
According to the requirements of $\mathcal{F}$, this leads to $\mathcal{I}(\mathcal{F}_2)\leq (t-1)|\mathcal{F}_2|^2+(r-p+\delta_0^2+\frac{1}{C_t}){n-t\choose k-t}^2$.

Since $\mathcal{I}(\mathcal{F})\geq \mathcal{I}(\mathcal{L}_{n,k,t}(|\mathcal{F}|))$, combining with this upper bound of $\mathcal{I}(\mathcal{F}_2)$, we have
\begin{align}\label{ineq19}
\mathcal{I}(\mathcal{F}_1)+2\mathcal{I}(\mathcal{F}_1,\mathcal{F}_2)&\geq \mathcal{I}(\mathcal{L}_{n,k,t}(|\mathcal{F}|))-\mathcal{I}(\mathcal{F}_2)\nonumber\\
&\geq (t-1)(|\mathcal{F}|^2-|\mathcal{F}_2|^2)+(p-\frac{1}{C_t}){n-t\choose k-t}^2\nonumber\\
&\geq p\big[(t-1)(2r+2\delta_0-p)+1-\frac{2}{C_t}\big]{n-t\choose k-t}^2.
\end{align}
On the other hand, due to the structure of $\mathcal{F}_1$, we also have
\begin{align}\label{ineq20}
\mathcal{I}(\mathcal{F}_1)+2\mathcal{I}(\mathcal{F}_1,\mathcal{F}_2)&\leq \sum_{i,j=1}^{p}\mathcal{I}(\mathcal{Y}_i,\mathcal{Y}_j)+2\sum_{i=1}^{p}\mathcal{I}(\mathcal{Y}_i,\mathcal{F}_2)\nonumber\\
&\leq (\sum_{i,j=1}^{p}|Y_i\cap Y_j|)\cdot{n-t\choose k-t}^2\cdot(1+\frac{1}{C_t})+\nonumber\\
&~~~2(\sum_{i=1}^{p}\sum_{F\in\mathcal{F}_2}|F\cap Y_i|)\cdot{n-t\choose k-t}\cdot(1+\frac{1}{C_t}).
\end{align}
Since $\sum_{i,j=1}^{p}|Y_i\cap Y_j|=tp+\sum_{i\neq j\in[p]}|Y_i\cap Y_j|$ and $\sum_{i=1}^{p}\sum_{F\in\mathcal{F}_2}|F\cap Y_i|)\leq p(t-1)|\mathcal{F}_2|$, we have
\begin{align*}
\sum_{i\neq j\in[p]}|Y_i\cap Y_j|\geq p(p-1)(t-1)(1-\frac{2rkp}{C_t}).
\end{align*}
According to the choice of $C_t$, $p(p-1)(t-1)(1-\frac{2rkp}{C_t})>p(p-1)(t-1)-1$. Thus, we have $\sum_{i\neq j\in[p]}|Y_i\cap Y_j|= p(p-1)(t-1)$. Therefore, $|Y_i\cap Y_j|=t-1$ for all $i\neq j\in [p]$. Moreover, by substituting $\sum_{i\neq j\in[p]}|Y_i\cap Y_j|= p(p-1)(t-1)$ into the above two inequalities about $\mathcal{I}(\mathcal{F}_1)+2\mathcal{I}(\mathcal{F}_1,\mathcal{F}_2)$, we have
\begin{align}\label{ineq20.5}
\sum_{i=1}^{p}\sum_{F\in\mathcal{F}_2}|F\cap Y_i|\geq p(t-1)(1-\frac{2rk}{C_t})|\mathcal{F}_2|.
\end{align}
Therefore, there exist at least $(1-\frac{2r^2kt}{C_t})|\mathcal{F}_2|$ $k$-sets in $\mathcal{F}_2$ satisfying $|F\cap Y_i|=t-1$ for all $i\in[p]$.

Moreover, (\ref{ineq20}) also leads to
\begin{align}\label{ineq21}
\mathcal{I}(\mathcal{F}_2)&\geq \mathcal{I}(\mathcal{L}_{n,k,t}(|\mathcal{F}|))-(\mathcal{I}(\mathcal{F}_1)+2\mathcal{I}(\mathcal{F}_1,\mathcal{F}_2))\nonumber\\
&\geq (t-1)|\mathcal{F}_2|^2+(r-p+\delta_0^2-\frac{4r^2kp^2}{C_t}){n-t\choose k-t}^2.
\end{align}
This completes the proof.
\end{proof}

Finally, we prove Lemma \ref{mainl08}.

\begin{proof}[Proof of Lemma \ref{mainl08}]
We prove this lemma by evaluating $\mathcal{I}(\mathcal{F})$ and $\mathcal{I}(\mathcal{F}_0)$ directly.

Since $\mathcal{L}_{n,k,t}^{(r)}\subseteq \mathcal{F}_0\subseteq \mathcal{L}_{n,k,t}^{(r+1)}$, we can assume $\mathcal{F}_0=\cup_{i=t}^{t+r}\mathcal{G}_i$, where $\mathcal{G}_i$ is a $t$-star with core $\{1,\ldots,t-1,i\}$ and $\mathcal{G}_i$ is a full $t$-star for each $t\leq i\leq t+r-1$. According to this structure, we have
\begin{align}\label{ineq23}
\mathcal{I}(\mathcal{F}_0)=\sum_{x\in [n]}|\mathcal{F}_0(x)|^2=(t-1)|\mathcal{F}|^2+r{n-t\choose k-t}^2+|\mathcal{F}_0(t+r)|^2+\sum_{x\in [t+r+1,n]}|\mathcal{F}_0(x)|^2.
\end{align}
Since $|\mathcal{F}_{0}(t+r)|={n-t\choose k-t}-(1-\delta){n-(t+r)\choose k-t}$ and $|\mathcal{F}_{0}(x)|> \sum_{i=t+1}^{t+r}{n-i\choose k-(t+1)}$ for $x\in [t+r+1,n]$, this leads to
\begin{align*}
\mathcal{I}(\mathcal{F}_{0})\geq& (t-1)|\mathcal{F}|^2+r{n-t\choose k-t}^2+(\delta{n-(t+r)\choose k-t}+\sum_{i=1}^{r}{n-(t+i)\choose k-(t+1)})^2+ \nonumber\\
&(n-t-r)(\sum_{i=1}^{r}{n-(t+i)\choose k-(t+1)})^2.
\end{align*}

On the other hand, denote $\tilde{\mathcal{H}}_j=\mathcal{H}_j\setminus{\mathcal{F}([t+1])}$, we have
\begin{align}\label{ineq24}
\mathcal{I}(\mathcal{F})&=\sum_{x\in [n]}|\mathcal{F}(x)|^2=\sum_{x\in [t+1]}(|\mathcal{F}|-|\tilde{\mathcal{H}}_x|)^2+\sum_{x\in [t+2,n]}|\mathcal{F}(x)|^2\nonumber\\
&=(t+1)|\mathcal{F}|^2-2|\mathcal{F}|(\sum_{x\in [t+1]}|\tilde{\mathcal{H}}_x|)+\sum_{x\in [t+1]}|\tilde{\mathcal{H}}_x|^2+\sum_{x\in [t+2,n]}|\mathcal{F}(x)|^2.
\end{align}
Since $|\mathcal{F}([t+1])|\leq {n-(t+1)\choose k-(t+1)}$, $|\mathcal{F}(x)|\leq (t+1){n-(t+1)\choose k-(t+1)}$ for $x\in [t+2,n]$ and $\sum_{x\in [t+1]}|\tilde{\mathcal{H}}_x|=|\mathcal{F}|-|\mathcal{F}([t+1])|$, thus
\begin{align*}
\mathcal{I}(\mathcal{F})&\leq (t-1)|\mathcal{F}|^2+\sum_{x\in [t+1]}|\tilde{\mathcal{H}}_x|^2+2|\mathcal{F}||\mathcal{F}([t+1])|+(n-t-1)(t+1)^2{n-(t+1)\choose k-(t+1)}^2.
\end{align*}
Noted that $|\tilde{\mathcal{H}}_x|\leq {n-(t+1)\choose k-t}$, by Theorem \ref{convex_property}, we have
\begin{align}\label{ineq25}
\sum_{x\in [t+1]}|\tilde{\mathcal{H}}_x|^2\leq r{n-(t+1)\choose k-t}^2+(|\mathcal{F}|-|\mathcal{F}([t+1])|-r{n-(t+1)\choose k-t})^2=m_0
\end{align}
and the inequality holds if and only if $|\tilde{\mathcal{H}}_x|={n-(t+1)\choose k-t}$ for exactly $r$ distinct $x$s in $[t+1]$ and the rest $k$-sets are all contained in another $\tilde{\mathcal{H}_x}$.

When $r<t$, since $|\mathcal{H}_j|\geq\frac{\delta}{C_1}{n-t\choose k-t}$ for each $j\in[t+1]$, we have $|\tilde{\mathcal{H}}_j|=|\mathcal{H}_j|-|\mathcal{F}([t+1])|\geq\frac{\delta}{2C_1}{n-t\choose k-t}$. By Theorem \ref{convex_property}, $\sum_{x\in [t+1]}|\tilde{\mathcal{H}}_x|^2\leq m_0-\frac{\delta^2}{4C_1}{n-t\choose k-t}^2$. Thus, according to the former evaluation, we have
\begin{align*}
\mathcal{I}(\mathcal{F}_{0})-\mathcal{I}(\mathcal{F})\geq& r{n-t\choose k-t}^2+(\delta{n-(t+r)\choose k-t}+\sum_{i=1}^{r}{n-(t+i)\choose k-(t+1)})^2-\\
&\sum_{x\in [t+1]}|\tilde{\mathcal{H}}_x|^2-2|\mathcal{F}|{n-(t+1)\choose k-(t+1)}-(n-t-1)(t+1)^2{n-(t+1)\choose k-(t+1)}^2\\
\geq& \frac{\delta^2}{4C_1}{n-t\choose k-t}^2-4kt^2{n-t\choose k-t}{n-(t+1)\choose k-(t+1)}.
\end{align*}
Since ${n-(t+1)\choose k-(t+1)}\leq \frac{{n-t\choose k-t}}{4kt^2C_t^3}$, we have $\mathcal{I}(\mathcal{F}_{0})>\mathcal{I}(\mathcal{F})$.

When $r=t$, first, we have the following claim.

\textbf{Claim 13.} $\mathcal{F}=\cup_{j=1}^{t+1}\mathcal{H}_j$ maximizes $\mathcal{I}(\mathcal{F})$ only if $\mathcal{F}$ contains $t$ full $t$-stars.
\begin{proof}

First, we show that $\mathcal{F}$ must contain $t$ almost full $t$-stars. Assume that $\mathcal{F}_1=\cup_{j=1}^{t+1}\mathcal{H}'_j$ contains $t$ full $t$-stars with size $|\mathcal{F}|$. Then, from (\ref{ineq24}) we have
\begin{align*}
\mathcal{I}(\mathcal{F}_1)&\geq(t-1)|\mathcal{F}|^2+m_1+2|\mathcal{F}|{n-(t+1)\choose k-(t+1)},
\end{align*}
where $m_1=t{n-(t+1)\choose k-t}^2+(|\mathcal{F}|-{n-(t+1)\choose k-(t+1)}-t{n-(t+1)\choose k-t})^2$.

W.l.o.g., assume that $|\tilde{\mathcal{H}}_{1}|\geq \ldots\geq |\tilde{\mathcal{H}}_{t+1}|$ and $(1-\frac{1}{C_t}){n-(t+1)\choose k-t}\geq |\tilde{\mathcal{H}}_{t}|\geq |\tilde{\mathcal{H}}_{t+1}|$. According to the size of $\mathcal{F}$, this indicates that $\delta\neq 1$. Since $\sum_{x\in [t+1]}|\tilde{\mathcal{H}}_x|=|\mathcal{F}|-|\mathcal{F}([t+1])|$, by Theorem \ref{convex_property}, we have
\begin{align*}
\sum_{x\in [t+1]}|\tilde{\mathcal{H}}_x|^2\leq& (t-1){n-(t+1)\choose k-t}^2+(1-\frac{1}{C_t})^2{n-(t+1)\choose k-t}^2+\\
&(|\mathcal{F}|-|\mathcal{F}([t+1])|-(t-\frac{1}{C_t}){n-(t+1)\choose k-t})^2\leq m_0-\frac{(1-\delta)}{C_t}{n-t\choose k-t}^2.
\end{align*}
Since $m_1-m_0\geq -2\delta{n-t\choose k-t}({n-(t+1)\choose k-(t+1)}-|\mathcal{F}([t+1])|)$, by the choices of $n$ and $\delta$, this leads to
\begin{align*}
\mathcal{I}(\mathcal{F}_1)-\mathcal{I}(\mathcal{F})\geq& \frac{1-\delta}{C_t}{n-t\choose k-t}^2+2({n-(t+1)\choose k-(t+1)}-|\mathcal{F}([t+1])|)(|\mathcal{F}|-2\delta{n-t\choose k-t})\\
&-(n-t-1)(t+1)^2{n-(t+1)\choose k-(t+1)}^2>0,
\end{align*}
a contradiction. Therefore, to maximize $\mathcal{I}(\mathcal{F})$, $|\tilde{\mathcal{H}_{j}}|> (1-\frac{1}{C_t}){n-(t+1)\choose k-t}$ for $j\in[t]$.

Now, we prove that $\mathcal{F}$ must contain $t$ full $t$-stars. W.l.o.g., assume that $|\tilde{\mathcal{H}}_{1}|\geq \ldots\geq |\tilde{\mathcal{H}}_{t+1}|$ and $|\mathcal{H}_{t}|<{n-t\choose k-t}$. Then, there is an $F\notin \mathcal{H}_{t}$ containing $[t+1]\setminus\{t\}$. Pick some $G_0\in \tilde{\mathcal{H}}_{t+1}$ and replace $G_0$ with $F$. Denote the resulting new family as $\mathcal{F}'$, we have
\begin{align*}
\sum_{x\in F}|\mathcal{F}'(x)|-\sum_{x\in G_0}|\mathcal{F}(x)|&\geq (|\mathcal{F}|-|\tilde{\mathcal{H}}_{t+1}|)-(|\mathcal{F}|-|\tilde{\mathcal{H}}_{t}|)-k(t+1){n-(t+1)\choose k-(t+1)}\\
&\geq |\tilde{\mathcal{H}}_{t}|-|\tilde{\mathcal{H}}_{t+1}|-\frac{1}{C_t^2}{n-t\choose k-t}\geq (\frac{1}{C_1}-\frac{2}{C_t}){n-(t+1)\choose k-t}>0.
\end{align*}
This shows that as long as $\mathcal{F}$ contains less than $t$ full $t$-stars, we can get a new family $\mathcal{F}'$ with $\mathcal{I}(\mathcal{F}')>\mathcal{I}(\mathcal{F})$. Therefore, $\mathcal{F}$ must contain $t$ full $t$-stars.
\end{proof}
Now, assume that $\mathcal{F}$ contains $t$ full $t$-stars. Our aim is to find a proper $\mathcal{F}_0$ with the required structure satisfying $\mathcal{I}(\mathcal{F}_0)>\mathcal{I}(\mathcal{F})$.

According to the structure of $\mathcal{F}$, we have $m_0=m_1$ and
\begin{align}\label{ineq26}
\mathcal{I}(\mathcal{F})&=(t-1)|\mathcal{F}|^2+2|\mathcal{F}|{n-(t+1)\choose k-(t+1)}+m_0+\sum_{x\in [t+2,n]}|\mathcal{F}(x)|^2.
\end{align}
Meanwhile,
\begin{align}\label{ineq27}
\mathcal{I}(\mathcal{F}_0)=(t-1)|\mathcal{F}|^2+t{n-t\choose k-t}^2+({n-t\choose k-t}-(1-\delta){n-2t\choose k-t})^2+\sum_{x\in [2t+1,n]}|\mathcal{F}_0(x)|^2.
\end{align}
Thus,
\begin{align*}
\mathcal{I}(\mathcal{F}_0)-\mathcal{I}(\mathcal{F})=&t{n-t\choose k-t}^2+({n-t\choose k-t}-(1-\delta){n-2t\choose k-t})^2-m_0-2|\mathcal{F}|{n-(t+1)\choose k-(t+1)}\\
&+\sum_{x\in [2t+1,n]}|\mathcal{F}_0(x)|^2-\sum_{x\in [t+2,n]}|\mathcal{F}(x)|^2.
\end{align*}
Denote $\Delta_1={n-t\choose k-t}-{n-2t\choose k-t}=\sum_{i=t+1}^{2t}{n-i\choose k-(t+1)}$ and $\Delta_2=\sum_{i=1}^{t-2}i{n-(2t-i)\choose k-(t+1)}$. Noted that
\begin{align*}
|\mathcal{F}|-{n-(t+1)\choose k-(t+1)}-t{n-(t+1)\choose k-t}&=\delta {n-2t\choose k-t}+\sum_{i=t+2}^{2t-1}({n-i\choose k-t}-{n-(t+1)\choose k-t})\\
&=\delta {n-2t\choose k-t}-\sum_{i=1}^{t-2}i{n-(2t-i)\choose k-(t+1)}=\delta {n-2t\choose k-t}-\Delta_2.
\end{align*}
Thus, we have
\begin{align}\label{ineq28}
\mathcal{I}(\mathcal{F}_0)-\mathcal{I}(\mathcal{F})=&t({n-t\choose k-t}+{n-(t+1)\choose k-t}){n-(t+1)\choose k-(t+1)}-2|\mathcal{F}|{n-(t+1)\choose k-(t+1)}+\nonumber\\
&(\delta {n-2t\choose k-t}+\Delta_1)^2-(\delta {n-2t\choose k-t}-\Delta_2)^2+\sum_{x\in [2t+1,n]}|\mathcal{F}_0(x)|^2-\sum_{x\in [t+2,n]}|\mathcal{F}(x)|^2\nonumber\\
=&(2\Delta_2+(t-2){n-(t+1)\choose k-(t+1)}){n-(t+1)\choose k-(t+1)}+2\delta{n-2t\choose k-t}(\Delta_1+\Delta_2-{n-(t+1)\choose k-(t+1)})+\nonumber\\
&\Delta_1^2-\Delta_2^2+\sum_{x\in [2t+1,n]}|\mathcal{F}_0(x)|^2-\sum_{x\in [t+2,n]}|\mathcal{F}(x)|^2.
\end{align}

Since for each $x\in [2t+1,n]$, $|\mathcal{F}_0(x)|=\sum_{i=t+1}^{2t}{n-i\choose k-(t+1)}+|\mathcal{G}_{2t}(x)|$ and for each $x\in [t+2,n]$, $|\mathcal{F}(x)|={n-(t+2)\choose k-(t+2)}+t{n-(t+2)\choose k-(t+1)}+|\tilde{\mathcal{H}}_{t+1}(x)|$. Thus,
\begin{align*}
|\mathcal{F}_0(x)|-|\mathcal{F}(x)|\geq |\mathcal{G}_{2t}(x)|-|\mathcal{H}_{t+1}(x)|-t^2{n-(t+2)\choose k-(t+2)}
\end{align*}
for each $x\in [2t+1,n]$. Since $\sum_{x\in [t+2,2t]}|\mathcal{F}(x)|^2\leq t(t+1)^2{n-(t+1)\choose k-(t+1)}^2$, therefore, we only need to find a proper $\mathcal{G}_{2t}$ to control $|\mathcal{G}_{2t}(x)|-|\mathcal{H}_{t+1}(x)|$ for all $x\in [2t+1,n]$. Let $\mathcal{H}_{t+1}=\mathcal{H}_{t+1}^{1}\sqcup\mathcal{H}_{t+1}^2$, where $\mathcal{H}_{t+1}^{1}=\{F\in \mathcal{H}_{t+1}: F\cap [t+2,2t]=\emptyset\}$ and $\mathcal{H}_{t+1}^2=\mathcal{H}_{t+1}\setminus\mathcal{H}_{t+1}^1$. Noted that $|\mathcal{G}_{2t}|\geq|\mathcal{H}_{t+1}|$, thus for each $F\in \mathcal{H}_{t+1}^1$, we can arrange a $G(F)=\{1,\ldots,t-1,2t\}\cup (F\setminus[t])$ in $\mathcal{G}_{2t}$. Denote $\mathcal{G}_{2t}^1=\{G(F): F\in \mathcal{H}_{t+1}^1\}$ and $\mathcal{G}_{2t}^2=\mathcal{G}_{2t}\setminus\mathcal{G}_{2t}^1$. Clearly, for $x\in [2t+1,n]$, $|\mathcal{G}_{2t}^1(x)|=|\mathcal{H}_{t+1}^1(x)|$. Since for each $F'\in\mathcal{H}_{t+1}^2$, $([t]\cup\{i_0\})\subseteq F'$ for some $i_0\in [t+2,2t]$. Therefore, for $x\in [2t+1,n]$ we have
\begin{align*}
|\mathcal{G}_{2t}(x)|-|\mathcal{H}_{t+1}(x)|=|\mathcal{G}_{2t}^2(x)|-|\mathcal{H}_{t+1}^2(x)|\geq -(t-1){n-(t+2)\choose k-(t+2)}.
\end{align*}
This indicates that
\begin{align}\label{ineq29}
\sum_{x\in [2t+1,n]}|\mathcal{F}_0(x)|^2-\sum_{x\in [t+2,n]}|\mathcal{F}(x)|^2&=\sum_{x\in [2t+1,n]}(|\mathcal{F}_0(x)|-|\mathcal{F}(x)|)(|\mathcal{F}_0(x)|+|\mathcal{F}(x)|)-\sum_{x\in [t+2,2t]}|\mathcal{F}(x)|^2\nonumber\\
&\geq -2t(t+1)^2n{n-(t+2)\choose k-(t+2)}{n-(t+1)\choose k-(t+1)}-t(t+1)^2{n-(t+1)\choose k-(t+1)}^2\nonumber\\
&\geq -4t(t+1)^2k{n-(t+1)\choose k-(t+1)}^2.
\end{align}
Since $t\geq 2$, based on the choice of $\delta,n$ and (\ref{ineq29}), we have
\begin{align*}
\mathcal{I}(\mathcal{F}_0)-\mathcal{I}(\mathcal{F})&\geq 2\delta{n-2t\choose k-t}{n-(t+2)\choose k-(t+1)}-5t(t+1)^2k{n-(t+1)\choose k-(t+1)}^2>0.
\end{align*}
This completes the proof.
\end{proof}

\section{Concluding remarks and open problems}

In this paper, we introduce a new intersection problem which can be viewed as an inverse problem of Erd\H{o}s-Ko-Rado type theorems. This problem concerns the extremal structure of the $k$-uniform family of subsets that maximizes the total intersection among all families with the same size. Using the quantitative shifting method, we provide two structural characterizations of the optimal family satisfying $\mathcal{I}(\mathcal{F})=\mathcal{MI}(\mathcal{F})$ for several ranges of $|\mathcal{F}|$. To some extent, these results determine the unique structure of the optimal family and characterize the relation between maximizing $\mathcal{I}(F)$ and being intersecting. However, there are several limits of our results that may require some further research.

\begin{itemize}
  \item One can remove the uniformity requirement of the family in Question \ref{question} and consider a more general question:
      \begin{question}\label{question2}
      For a family of subsets $\mathcal{F}\subseteq 2^{[n]}$, if $\mathcal{I}(\mathcal{F})=\mathcal{MI}(\mathcal{F})$, what can we say about its structure?
      \end{question}
      It should be noted that this question is highly related to Ahlswede-Katona¡¯s \cite{AK1977} average distance problem in Hamming space: For every $1\leq M\leq 2^n$, determine the minimum average Hamming distance $D_n(M)$ of binary codes with length $n$ and size $M$. Based on the correspondence between binary vectors with length $n$ and subsets of $[n]$, for $|\mathcal{F}|=M$, there is a qualitative relation between $D(M)$ and $\mathcal{MI}(\mathcal{F})$. And this qualitative relation becomes an equivalence when we consider both problems for $k$-uniform families (i.e., codes with constant weight $k$). Over the years, there are a number of papers dealing with this topic. Alth\"{o}fer and Sillke \cite{AS1992}, Fu, together with other authors (see \cite{XF1998,FKS1999,FWY2001,XFJ2008}), Mounits \cite{Mounits2008}, as well as Yu and Tan \cite{YT2019}, proved various of bounds on $D_n(M)$. In view of $D_n(M)$ for codes with constant weight $k$, Corollary \ref{coro2} and Corollary \ref{coro4} actually provide better lower bounds on $D_n(M)$ for the required ranges of $M$ compared to the results in \cite{XFJ2008}.
  \item The method we use for the proof of Theorem \ref{main0_t=1} is the quantitative shifting arguments introduced by Das, Gan and Sudakov in \cite{DGS2016}. While for the proof of Theorem \ref{main0}, we do a lot of modifications about this method that involve analysis of degrees of $s$-sets $(1\leq s\leq t)$ in $\mathcal{F}$ from different level. This requires $n$ to be larger than a certain polynomial of $r$. As a consequence, our results cannot cover the whole range of $M$ from $1$ to ${n\choose k}$.

      Maybe due to the nature of the problem itself, the implementation of this method requires a great deal of countings and evaluations, which might cover the idea and intuition for dealing with this kind of problems. Therefore, as one direction for the further study, one can try to use other methods to obtain stronger results and reduce $n$'s polynomially dependent of $r$.
  \item As two major extensions, Erd\H{o}s-Ko-Rado type intersection problems over permutations and vector spaces draw a lot of attentions these years. Since one can easily extend the notion of total intersection under these settings, therefore, it is natural to ask Question \ref{question} for families of permutations and vector spaces.
  \item Given a hypergraph $\mathcal{H}$ with vertex set $V$, for every $v\in V$, denote $\deg_{\mathcal{H}}(v)$ as the degree of $v$ in $\mathcal{H}$. Since families of subsets are often viewed as hypergraphs, therefore in view of hypergraphs, Question \ref{question} actually asks the structure of the extremal hypergraph which maximizes the value of $\sum_{v\in V}\deg_{\mathcal{H}}(v)^2$ with $|\mathcal{H}|$ fixed. If we treat $|\mathcal{H}|$ as some kind of perimeter and $\sum_{v\in V}\deg_{\mathcal{H}}(v)^2$ as the area, Question \ref{question} can be viewed as an isoperimetric problem for $k$-uniform hypergraphs. There are already some works concerning isoperimetric inequalities about $n$-dimensional Boolean cube \footnote{This inequality was originally proved by Harper \cite{Harper1964}, Lindsey \cite{Lindsey1964}, Bernstein \cite{Bernstein1967} and Hart \cite{Hart1976}, see Theorem 1.1 in \cite{EKL19}.}, see \cite{EKL18}, \cite{EKL19} and references therein. In view of this, is there any connections between Question \ref{question} and the isoperimetric inequality?
\end{itemize}

\subsection*{Acknowledgements}

The authors would like to thank Binchen Qian and Yuanxiao Xi for their helpful discussions about this topic and their suggestions in the presentation of the proof.

\bibliographystyle{abbrv}
\bibliography{On_an_inverse_problem_of_the_Erdos-Ko-Rado_type_theorems}

\begin{thebibliography}{10}

\bibitem{AK1977}
R.~Ahlswede and G.~O.~H. Katona.
\newblock Contributions to the geometry of {H}amming spaces.
\newblock {\em Discrete Math.}, 17(1):1--22, 1977.

\bibitem{AS1992}
I.~Alth\"{o}fer and T.~Sillke.
\newblock An ``average distance'' inequality for large subsets of the cube.
\newblock {\em J. Comb. Theory, Ser. B}, 56(2):296--301, 1992.

\bibitem{BDDLS2015}
J.~Balogh, S.~Das, M.~Delcourt, H.~Liu, and M.~Sharifzadeh.
\newblock Intersecting families of discrete structures are typically trivial.
\newblock {\em J. Comb. Theory, Ser. A}, 132:224--245, 2015.

\bibitem{BDLST2019}
J.~Balogh, S.~Das, H.~Liu, M.~Sharifzadeh, and T.~Tran.
\newblock Structure and supersaturation for intersecting families.
\newblock {\em Electron. J. Combin.}, 26(2):Paper 2.34, 38, 2019.

\bibitem{Bernstein1967}
A.~Bernstein.
\newblock Maximally connected arrays on the $n$-cube.
\newblock {\em SIAM J. Appl. Math.}, 15:1485--1489, 1967.

\bibitem{CK03}
P.~J. Cameron and C.~Y. Ku.
\newblock Intersecting families of permutations.
\newblock {\em European J. Combin.}, 24(7):881--890, 2003.

\bibitem{CP2010}
A.~Chowdhury and B.~Patk\'{o}s.
\newblock Shadows and intersections in vector spaces.
\newblock {\em J. Comb. Theory, Ser. A}, 117:1095--1106, 2010.

\bibitem{DGS2016}
S.~Das, W.~Gan, and B.~Sudakov.
\newblock The minimum number of disjoint pairs in set systems and related
  problems.
\newblock {\em Combinatorica}, 36(6):623--660, 2016.

\bibitem{DT2016}
S.~Das and T.~Tran.
\newblock Removal and stability for {E}rd{\H{o}}s-{K}o-{R}ado.
\newblock {\em SIAM J. Discrete Math.}, 30(2):1102--1114, 2016.

\bibitem{Ellis2011}
D.~Ellis.
\newblock Stability for {$t$}-intersecting families of permutations.
\newblock {\em J. Comb. Theory, Ser. A}, 118(1):208--227, 2011.

\bibitem{Ellis12}
D.~Ellis.
\newblock Setwise intersecting families of permutations.
\newblock {\em J. Comb. Theory, Ser. {A}}, 119(4):825--849, 2012.

\bibitem{Ellis14}
D.~Ellis.
\newblock Forbidding just one intersection, for permutations.
\newblock {\em J. Comb. Theory, Ser. {A}}, 126:136--165, 2014.

\bibitem{EFP2011}
D.~Ellis, E.~Friedgut, and H.~Pilpel.
\newblock Intersecting families of permutations.
\newblock {\em J. Amer. Math. Soc.}, 24(3):649--682, 2011.

\bibitem{EKL18}
D.~Ellis, N.~Keller, and N.~Lifshitz.
\newblock On the structure of subsets of the discrete cube with small edge
  boundary.
\newblock {\em Discrete Anal.}, Paper No.9, 29pp, 2018.

\bibitem{EKL19}
D.~Ellis, N.~Keller, and N.~Lifshitz.
\newblock On a biased edge isoperimetric inequality for the discrete cube.
\newblock {\em J. Comb. Theory, Ser. {A}}, 163:118--162, 2019.

\bibitem{EKR}
P.~Erd{\H{o}}s, C.~Ko, and R.~Rado.
\newblock Intersection theorems for systems of finite sets.
\newblock {\em Quart. J. Math. Oxford Ser. (2)}, 12:313--320, 1961.

\bibitem{FF84}
P.~Frankl and Z.~F{\"{u}}redi.
\newblock On hypergraphs without two edges intersecting in a given number of
  vertices.
\newblock {\em J. Comb. Theory, Ser. {A}}, 36(2):230--236, 1984.

\bibitem{FKR2016}
P.~Frankl, Y.~Kohayakawa, and V.~R\"{o}dl.
\newblock A note on supersaturated set systems.
\newblock {\em European J. Combin.}, 51:190--199, 2016.

\bibitem{FK17}
P.~Frankl and A.~Kupavskii.
\newblock Uniform s-cross-intersecting families.
\newblock {\em Combin. Probab. Comput.}, 26(4):517--524, 2017.

\bibitem{FOT96}
P.~Frankl, K.~Ota, and N.~Tokushige.
\newblock Exponents of uniform l-systems.
\newblock {\em J. Comb. Theory, Ser. {A}}, 75(1):23--43, 1996.

\bibitem{FT92}
P.~Frankl and N.~Tokushige.
\newblock Some best possible inequalities concerning cross-intersecting
  families.
\newblock {\em J. Comb. Theory, Ser. {A}}, 61(1):87--97, 1992.

\bibitem{FT2016}
P.~Frankl and N.~Tokushige.
\newblock Invitation to intersection problems for finite sets.
\newblock {\em J. Comb. Theory, Ser. {A}}, 144:157--211, 2016.

\bibitem{FW1986}
P.~Frankl and R.~Wilson.
\newblock The {E}rd{\H{o}}s-{K}o-{R}ado theorem for vector spaces.
\newblock {\em J. Comb. Theory, Ser. A}, 43:228--236, 1986.

\bibitem{FKS1999}
F.-W. Fu, T.~Kl{\o}ve, and S.-Y. Shen.
\newblock On the {H}amming distance between two i.i.d. random {$n$}-tuples over
  a finite set.
\newblock {\em IEEE Trans. Inform. Theory}, 45(2):803--807, 1999.

\bibitem{FWY2001}
F.-W. Fu, V.~K. Wei, and R.~W. Yeung.
\newblock On the minimum average distance of binary codes: linear programming
  approach.
\newblock {\em Discrete Appl. Math.}, 111(3):263--281, 2001.

\bibitem{GS2020}
A.~Gir\~{a}o and R.~Snyder.
\newblock Disjoint pairs in set systems with restricted intersection.
\newblock {\em European J. Combin.}, 83:102998, 13pp, 2020.

\bibitem{Harper1964}
L.~Harper.
\newblock Optimal assignments of numbers to vertices.
\newblock {\em SIAM J. Appl. Math.}, 12:131--135, 1964.

\bibitem{Hart1976}
L.~Harper.
\newblock A note on the edges of the $n$-cube.
\newblock {\em Discrete Math.}, 14:157--163, 1976.

\bibitem{KKK2012}
G.~O.~H. Katona, G.~Y. Katona, and Z.~Katona.
\newblock Most probably intersecting families of subsets.
\newblock {\em Combin. Probab. Comput.}, 21(1-2):219--227, 2012.

\bibitem{Lindsey1964}
J.~Lindsey.
\newblock Assignment of numbers to vertices.
\newblock {\em Amer. Math. Monthly}, 71:508--516, 1964.

\bibitem{Mounits2008}
B.~Mounits.
\newblock Lower bounds on the minimum average distance of binary codes.
\newblock {\em Discrete Math.}, 308(24):6241--6253, 2008.

\bibitem{MR2014}
D.~Mubayi and V.~R{\"o}dl.
\newblock Specified intersections.
\newblock {\em Trans. Amer. Math. Soc.}, 366(1):491--504, 2014.

\bibitem{Rockafellar}
R.~Rockafellar.
\newblock {\em Convex Analysis}.
\newblock Princeton Landmarks in Mathematics and Physics. Princeton University
  Press, 1970.

\bibitem{Russell2012}
P.~A. Russell.
\newblock Compressions and probably intersecting families.
\newblock {\em Combin. Probab. Comput.}, 21(1-2):301--313, 2012.

\bibitem{Snevily03}
H.~S. Snevily.
\newblock A sharp bound for the number of sets that pairwise intersect at
  \emph{k} positive values.
\newblock {\em Combinatorica}, 23(3):527--533, 2003.

\bibitem{WZ11}
J.~Wang and H.~Zhang.
\newblock Cross-intersecting families and primitivity of symmetric systems.
\newblock {\em J. Comb. Theory, Ser. {A}}, 118(2):455--462, 2011.

\bibitem{WZ13}
J.~Wang and H.~Zhang.
\newblock Nontrivial independent sets of bipartite graphs and
  cross-intersecting families.
\newblock {\em J. Comb. Theory, Ser. {A}}, 120(1):129--141, 2013.

\bibitem{XF1998}
S.~Xia and F.~Fu.
\newblock On the average {H}amming distance for binary codes.
\newblock {\em Discrete Appl. Math.}, 89(1-3):269--276, 1998.

\bibitem{XFJ2008}
S.-T. Xia, F.-W. Fu, and Y.~Jiang.
\newblock On the minimum average distance of binary constant weight codes.
\newblock {\em Discrete Math.}, 308(17):3847--3859, 2008.

\bibitem{YT2019}
L.~Yu and V.~Y.~F. Tan.
\newblock An improved linear programming bound on the average distance of a
  binary code.
\newblock {\em arXiv:1910.09416v1}, 2019.

\end{thebibliography}
\end{document}